%% file: Moeller_Ullrich_Rev.tex
\documentclass[12pt,reqno,a4paper]{article}

\input{macros}

\begin{document}

\title{$L_2$-norm sampling discretization and recovery of functions from RKHS with finite trace}
\author{Moritz Moeller, Tino Ullrich\footnote{Corresponding author: tino.ullrich@mathematik.tu-chemnitz.de} \\\\
TU Chemnitz, Faculty of Mathematics, 09107 Chemnitz, Germany}

\maketitle

\begin{abstract}
\sloppy In this paper we study $L_2$-norm sampling discretization and sampling recovery of complex-valued functions in RKHS on $D \subset \R^d$ based on random function samples. We only assume the finite trace of the kernel (Hilbert-Schmidt embedding into $L_2$) and provide several concrete estimates with precise constants for the corresponding worst-case errors. In general, our analysis does not need any additional assumptions and also includes the case of non-Mercer kernels and also non-separable RKHS. The fail probability is controlled and decays polynomially in $n$, the number of samples. Under the mild additional assumption of separability we observe improved rates of convergence related to the decay of the singular values.  Our main tool is a spectral norm concentration inequality for infinite complex random matrices with independent rows complementing earlier results by Rudelson, Mendelson, Pajor, Oliveira and Rauhut. \newline

\medskip

\small%

\end{abstract}

\section{Introduction} 
This paper can be seen as a continuation of \cite{KUV19,KrUl19}. We study the reconstruction of complex-valued multivariate functions on a domain $D \subset \R^d$ from values at the (randomly sampled) nodes $\bX :=(\bx^1,...,\bx^n) \subset D$ via weighted least squares algorithms. In addition, we are interested in the sampling discretization of the squared $L_2$-norm of such functions using $n$ random nodes. Both problems recently gained substantial interest, see \cite{KrUl19, KUV19, Ul20, KrUl20, T2018, Te18, Te20_2}, and are strongly related as we know from Wasilkowski \cite{Wa84} and the recent systematic studies by Temlyakov \cite{Te18} and Gr\"ochenig \cite{Gr19} on $L_p$-norm discretization. Our main interest is on accurate estimates for worst-case errors depending on the number $n$ of nodes. In this paper, the functions are modeled as elements from some reproducing kernel Hilbert space $H(K)$, which is supposed to be compactly embedded into $L_2(D,\varrho_D)$. Its kernel is a positive definite Hermitian function $K\colon D\times D\to \C$. In the papers \cite{KrUl19, KUV19, NSU20} authors mainly restrict to the case of separable RKHS \cite{KUV19,NSU20} or Mercer kernels on compact domains \cite{KrUl19} with finite trace property to study the quantity
\begin{equation}\label{f00a}
	\sup\limits_{\|f\|_{H(K)} \leq 1} \int_D|f(\bx)-S^m_{\bX}f(\bx)|^2\,d\varrho_D(\bx)\,
\end{equation}
for some recovery operator $S^m_{\bX}$. It computes a best least squares fit $S_{\bX}^mf$ to the given data $$\mathbf{f} = (f(\bx^1),...,f(\bx^n))^\top$$ from the finite-dimensional space spanned by the first $m-1$ singular vectors $\eta_1(\cdot),...,\eta_{m-1}(\cdot)$ of the embedding  
\begin{equation}\label{f000}
    \Id : H(K) \to L_2(D,\varrho_D)\,.
\end{equation}
We complement existing results by a refined analysis based on spectral norm concentration of infinite matrices to improve on the constants and the bounds for the failure probability on the one hand. On the other hand, the question remained whether the bounds on \eqref{f00a} may be extended to the most general situation where only the finite trace condition is assumed. This setting is not covered by the above mentioned references. In this paper we construct a new (weighted) least squares algorithm for this general situation, which has been first addressed by Wasilkowski, Wo{\'z}niakowski in \cite{WaWo01}. Surprisingly, we were able to improve on the bound in \cite[Thm.\ 1]{WaWo01} by obtaining the worst-case bound $o(\sqrt{\log n / n})$ in case of square summable singular values $(\sigma_k)_k$ (finite trace) of the embedding. It seems that, in general, their decay influences the bounds rather weakly (in contrast to the results in \cite{KrUl19,KUV19,NSU20}).

In addition to the general sampling recovery problem we study the discretization of $L_2$-integral norms in reproducing kernel Hilbert spaces $H(K)$ where only random information is used. To be more precise, we provide bounds for the following $L_2$-worst-case discretization errors
\begin{equation}\label{discr_error}
	\sup\limits_{\|f\|_{H(K)} \leq 1} \Big|\int_D |f(\bx)|^2d\varrho_D(\bx) - \frac{1}{n}\sum\limits_{i=1}^n|f(\bx^i)|^2\Big|\,.
\end{equation}
This quantity controls the simultaneous discretization of the squared $L_2(D,\varrho_D)$-norms of all functions from $H(K)$. For finite-dimensional spaces we speak of Marcinkiewicz-Zygmund inequalities, a classical topic which also gained a lot of interest in recent years, see Temlyakov \cite{Te18} and the references therein. Let us emphasize that both problems (sampling recovery and discretization) are strongly related. It has been shown by Wasilkowski \cite{Wa84} that the recovery of the norm $\|\cdot \|$ of a function from a function class is equally difficult as the recovery of the function in that norm using linear information. In other words, if we have a good sampling recovery operator $Sf$ in $L_2(D,\varrho_D)$ we may construct an equally good recovery for the norm of $f$ by simply taking $\|Sf\|_{L_2(D,\varrho_D)}$ as approximant. This, however, is a simple consequence of the triangle inequality. Wasilkowski shows even more, namely that optimal information for the recovery problem is nearly optimal for the ``norm-recovery'' problem. However, let us emphasize that we recover the square of the norm in \eqref{discr_error} (rather than the norm itself). It has been observed by V.N. Temlyakov in \cite{T2018} that this indeed makes a difference if we assume a certain algebra property for point-wise multiplication, namely $\|fg\|_{H(K)} \leq c\|f\|_{H(K)}\cdot \|g\|_{H(K)}$ which is for instance present for mixed Sobolev spaces with smoothness $s>1/2$. Taking into account that in this framework optimal quadrature behaves asymptotically better than sampling recovery (the improvement happens in the $\log$), see \cite[Chap.\ 5, 9]{DuTeUl19} and the references therein, we see that Wasilkowski's result does not hold true for this slightly modified framework. In fact, the worst-case error \eqref{discr_error} may behave much better than the corresponding optimal sampling recovery error. In contrast to that we use random information here, i.e, nodes which are randomly drawn according to the natural (probability measure) $\varrho_D$ or some related measure and aim for results with high probability. As stated  below we obtain a less good asymptotic error behavior for the classical discretization operator in \eqref{discr_error} compared to the (non-squared) sampling recovery error in \eqref{f00a}. However, we are able to control the dependence on the parameters and the failure probability rather explicit as \eqref{disc_sep}, \eqref{f22c}, Theorem \ref{non-sep} and Theorem \ref{intr1} show. 

Major parts of the analysis in this paper are based on the following concentration inequality for sums of complex self adjoint (infinite) random matrices.

\begin{theorem}[Sect. \ref{sec_prob2}]\label{int_1}
Let \(\by^i , i= 1 \dots n \), be i.i.d random sequences from the complex \( \ell_2\). Let further $n \geq 3$ and  $M>0$ such that \(\| \by^i \|_2 \leq M\)  almost surely for all $i=1 \dots n$. We further put \({\bf \Lambda} := \Ept  \by^i \otimes  \by^i\) and assume that $\|{\bf \Lambda}\|_{2\to 2} \leq 1$. Then we have for $0<t\leq 1$
\[ \Prob \Big( \Big\| \frac{1}{n} \sum_{i=1}^n  \, \by^i \otimes  \by^i - {\bf \Lambda} \Big\|_{2 \to 2} \geq t \Big) \leq 2^\frac{3}{4} n\exp\Big(-\frac{t^2n}{21M^2}\Big)\,.\]
\end{theorem}
Finite-dimensional results of this type are given by Rudelson \cite{Ru99}, Tropp \cite{Tr11}, Oliveira \cite{Ol10}, Rauhut \cite{Ra10} and others. Mendelson and Pajor \cite{PaMe06} were the first who addressed the infinite-dimensional case of real matrices as well, see Remark \ref{compact}. The technique used has been introduced by Buchholz \cite{Bu01,Bu05} and further developed by Rauhut for the purpose of analyzing RIP matrices based on complex bounded orthonormal systems (see \cite{Ra10} and the references therein). It is based on an operator version of the non-commutative Khintchine inequality \cite{Bu01, Bu05} together with Talagrand's symmetrization technique. 

As a direct consequence of Theorem \ref{int_1} we obtain for separable $H(K)$ and a probability measure $\varrho_D$ always
\begin{equation}\label{disc_sep}
  \begin{split}
	&\Prob\Big( \sup\limits_{\|f\|_{H(K)}\leq 1}\Big|\int_D |f(\bx)|^2\,d\varrho_D(\bx)-\frac{1}{n}\sum\limits_{i=1}^n |f(\bx^i)|			^2\Big|>t\|\Id\|_{K,2}^2\Big)\\
	&~~~ \leq 2^{3/4}n\exp\Big(-\frac{t^2n \|\Id\|_{K,2}^2}{21\|K\|_{\infty}^2}\Big)
  \end{split}	
\end{equation}
if the kernel is bounded, i.e., $\|K\|_{\infty} := \sup_{\bx \in D} \sqrt{K(\bx,\bx)}<\infty$ (uniform boundedness).  This condition is equivalent to the fact that the embedding of $H(K)$ into $\ell_\infty$ is continuous and has norm less or equal a finite number $M$ (commonly called $M$-boundedness). The measure $\varrho_D$ is supposed to be a probability measure and $\bX = (\bx^1,...,\bx^n)$ are drawn independently at random according to $\varrho_D$. Note, that this problem is related to classical uniform bounds on the ``defect function'' in learning theory with respect to $M$-bounded function classes, see, e.g., \cite{CuckerZhou}, \cite{StChr08}. There, bounds for \eqref{discr_error} are usually given in terms of covering (or entropy) numbers of the unit ball of $H(K)$ in $\ell_\infty$, see \cite{KoTe}, \cite{StChr08}. Here we consider situations where we neither have such information nor an embedding into $\ell_\infty$. Choosing $t$ appropriately (see Theorem \ref{f1}), the worst-case discretization error may be bounded as $\mathcal{O}(\sqrt{(\log n)/n})$ with high probability. To get rid of the uniform boundedness condition of the function class we may work with the weaker finite trace condition 
\begin{equation}\label{f00}
  \trace{K} := \int_{D} K(\bx,\bx)d\varrho_D(\bx) < \infty
\end{equation}
and prove a similar error bound for a slightly modified discretization operator when sampling the nodes $\bx^i$ independently according to the modified measure $\nu(\bx)d\varrho_D(\bx)$ with $\nu(\bx) := K(\bx,\bx)/\trace{K}$\,. One only has to replace $\|K\|_{\infty}^2$ by 
$\trace{K}$ in the right-hand side of \eqref{disc_sep}. In other words, we have 
\begin{equation}\label{f22c}
	\sup\limits_{\|f\|_{H(K)}\leq 1}\left|\int_D |f(\bx)|^2\,d\varrho_D(\bx)-\frac{1}{n}\sum\limits_{i=1}^n \frac{|f(\bx^i)|^2}{\nu(\bx^i)}\right| \leq 
	\sqrt{21\trace{K}\|\Id\|^2r\frac{\log n}{n}}
\end{equation}
with probability exceeding $1-2n^{1-r}$ for large enough $n$, see Theorem \ref{thm_discr2}. This means that the success probability tends to $1$ rather quickly as the number of samples increases. 

%

As for the sampling recovery problem we start with a result in the most general situation. A modification of the recovery operator $\widetilde{S}_{\bX}^m$ from \cite{KrUl19,KUV19}, see Algorithm 1 below,  has been used to study the situation which is left as an open problem in \cite{KUV19}.
The result reads as follows.
\begin{theorem}[Sect.\ \ref{nonsepRKHS}]\label{non-sep}
Let $H(K)$ be a reproducing kernel Hilbert space on a subset $D \subset \R^d$ with a positive definite Hermitian kernel $K(\cdot,\cdot)$ such that the finite trace property \eqref{f00} holds true.
Let $r > 1$ and $m,n \in \N$, \(n \geq 3\), where \(m\) is chosen according to \eqref{achoice_m2}.
Drawing $\bX = (\bx^1,...,\bx^n)$ at random according to the product measure $(\varrho_m(\bx)d\varrho_D(\bx))^n$ with the density defined in \eqref{f101b}, we have 
$$ \sup_{\|f\|_{H(K)} \leq 1 } \big\| f - \widetilde{S}_{\bX}^m f \big\|^2_{L_2(D,\varrho_D)}  \leq 441  \max\Big\{\sigma_m^2, \frac{r \log n }{n} \sum_{j=m}^\infty \sigma_j^2, \frac{\operatorname{tr}_0(K)}{n}  \Big\} $$
with probability at least \( 1 - \eta n^{1-r}\)
where \(\eta = 2^\frac{3}{4} +1\). \(\widetilde{S}_{\bX}^m\) is the least squares operator from Algorithm \ref{algo1:reweighted} together with \eqref{f101b} and $\operatorname{tr}_0(K)$ is defined in \eqref{C_K}.
\end{theorem}

In fact, we recover all $f \in H(K)$ from sampled values at $\bX = (\bx^1,...,\bx^n)$ simultaneously with probability larger than $1-3n^{-r}$ by only assuming that the kernel $K(\cdot, \cdot)$ has finite trace \eqref{f00}. Note that this result improves on a result by Wasilkowski, Wo{\'z}niakowski \cite{WaWo01}, where only the finite trace is required, see also Novak, Wo{\'z}niakowski \cite[Thm.\ 26.10]{NoWoIII}. The authors proved (roughly speaking) a rate of $n^{-1/(2+p)}$ for the worst-case error with respect to standard information if the sequence of singular numbers is $p$-summable for $p \leq 2$. We refer to Section \ref{nonsepRKHS} for further explanation.

In order to define the recovery operator $\widetilde{S}^m_{\bX}$ and the sampling density $\varrho_m(\bx)$ we need to incorporate spectral properties of the embedding \eqref{f000}, namely also the left and right singular functions $(e_k)_k \subset H(K)$ and $(\eta_k)_k \subset L_2(D,\varrho_D)$ ordered according to their importance (size of the corresponding singular number). Both systems are orthonormal in the respective spaces related by $e_k = \sigma_k \eta_k$\,.

The above result can be improved essentially if we assume that $H(K)$ is separable. This is for instance the case if $K$ is a Mercer kernel, i.e., continuous on a bounded and compact domain $D$. However, assuming only separability of $H(K)$ also includes the situation of continuous kernels on unbounded domains $D$, even $D=\R^d$. The following result already improves on the result given in \cite{KrUl19}, \cite{KUV19} in several directions. The theorem works under less restrictive assumptions, the constants are improved and, last but not least, the failure probability decays polynomially in $n$. We would like to point that, while  preparing this manuscript, M. Ullrich \cite{Ul20} proved a version of the next theorem with stronger requirements and different constants based on Oliveira's concentration result (see Remark \ref{Oliv_comp}). The following theorem is a reformulation of Theorem \ref{sampling_numbers} in Section \ref{Sect5}.
 
\begin{theorem}[Sect.\ \ref{Sect5}]\label{intr1} Let $K$ be a positive definite Hermitian kernel such that $H(K)$ is separable and the finite trace condition \eqref{f00} holds true. With the notation from above we have for $n\in \N$ and 
\begin{equation} \label{achoice_m2}
    m := \left\lfloor\frac{n}{14 r\log n}\right\rfloor
\end{equation}
the bound
\[ \Prob \Big( \sup_{\|f\|_{H(K)} \leq 1} \| f - \widetilde{S}_{\bX}^m f\|_{L_2(D,\varrho_D)}^2 \leq \frac{15}{m}\sum\limits_{j=\lfloor m/2 \rfloor}^{\infty}\sigma_j^2\Big)\geq 1 -3 n^{1-r}\,, \]
where $\bX = (\bx^1,...,\bx^n)$ is sampled independently according to the product measure $(\varrho_m(\bx)d\varrho_D(\bx))^n$ (see \eqref{density}) and the operator $\widetilde{S}_{\bX}^m$ is defined in Algorithm 1. 
\end{theorem}
We would like to emphasize that the operator $\widetilde{S}_{\bX}^m$ uses $n \asymp m\log m$ samples of its argument. Based on this result it has been recently shown by the second named author (and coauthors, see \cite{NSU20}) that there exists a sampling operator $\widetilde{S}^m_J$ using only $\mathcal{O}(m)$ samples yielding the bound 
$$
  \sup_{\|f\|_{H(K)} \leq 1} \| f - \widetilde{S}_{J}^m f\|_{L_2(D,\varrho_D)}^2 \leq \frac{C\log(m)}{m}\sum\limits_{j=\lfloor cm \rfloor}^{\infty}\sigma_j^2
$$
with universal and specified constants $C,c>0$. However, for this improvement one has to sacrifice the high success probability.

\paragraph{Notation.} As usual $\N$ denotes the natural numbers, $\N_0:=\N\cup\{0\}$, $\Z$ denotes the integers, $\R$ the real numbers and \(\R_+\) the non-negative real numbers and $\C$ the complex numbers. If not indicated otherwise  $\log(\cdot)$ denotes the natural logarithm of its argument. $\C^n$ denotes the complex $n$-space, whereas $\C^{m\times n}$ denotes the set of all $m\times n$-matrices $\bL$ with complex entries. Vectors and matrices are usually typesetted boldface with $\bx,\by\in \C^n$. The matrix $\bL^{\ast}$ denotes the adjoint matrix. The spectral norm of matrices $\bL$ is denoted by $\|\bL\|$ or $\|\bL\|_{2\to 2}$.  For a complex (column) vector $\by\in \C^n$ (or $\ell_2$) we will often use the tensor notation for the matrix
$$
	\by \otimes \by := \by\cdot \by^\ast = \by \cdot \overline{\by}^\top 	
	\in\C^{n \times n}\;\textnormal{(or $\C^{\N\times\N}$)}\,.
$$
For $0<p\leq \infty$ and $\bx\in \C^n$ we denote $\|\bx\|_p := (\sum_{i=1}^n
|x_i|^p)^{1/p}$ with the usual modification in the case $p=\infty$ or $\bx$ being an infinite sequence. Operator norms for $T:H(K) \to L_2$ will be denoted with $\|T\|_{K,2}$. As usual we will denote with $\Ept X$ the expectation of a random variable $X$ on a probability space $(\Omega, \mathcal{A}, \Prob)$. Given a measurable subset $D\subset \R^d$ and a measure $\varrho$ we denote with $L_2(D,\varrho)$ the space of all square integrable complex-valued functions (equivalence classes) on $D$ with $\int_D |f(\bx)|^2\, d\varrho(\bx)<\infty$. We will often use $\Omega = D^n$ as probability space with the product measure $\Prob = d\varrho^n$ if $\varrho$ is a probability measure itself. We sometimes use the notation $f = \mathcal{O}(g)$ for positive functions $f,g$, which means that there is a constant $c>0$ with $f(t) \leq cg(t)$. In addition we say $f = o(g)$ if $f(t)/g(t) \to 0$ as $t \to \infty$.

\section{Concentration results for sums of random matrices}
\label{sect_prob}
Let us begin with concentration inequalities for the spectral norm of sums of complex rank-$1$ matrices. Such matrices appear as $\bL^\ast \bL$ when studying least squares solutions of over-determined linear systems
\begin{equation*} 
		\bL\cdot \mathbf{c} = \mathbf{f}\,,
\end{equation*}
where $\bL \in \C^{n \times m}$ is a matrix with $n>m$, $\mathbf{f} \in \C^n$ and $\mathbf{c} \in \C^{m-1}$. It is well-known that the above system may not have a solution. However, we can ask for the vector $\mathbf{c}$ which minimizes the residual $\|\mathbf{f}-\bL\cdot \mathbf{c}\|_2$. Multiplying the system with $\bL^\ast$ gives 
$$
	\bL^{\ast}\bL\cdot \mathbf{c} = \bL^{\ast}\cdot \mathbf{f}
$$
which is called the system of normal equations. If $\bL$ has full rank then the unique solution of the least squares problem is given by
$$
	\mathbf{c} = (\bL^{\ast}\bL)^{-1}\bL^{\ast}\cdot \mathbf{f}\,.
$$
For function recovery and discretization problems we will use the following matrix 
\begin{equation}\label{f0}
	\bL_m := \left(\begin{array}{cccc}
			\eta_1(\bx^1) &\eta_2(\bx^1)& \cdots & \eta_{m-1}(\bx^1)\\
			\vdots & \vdots && \vdots\\
			\eta_1(\bx^n) &\eta_2(\bx^n)& \cdots & \eta_{m-1}(\bx^n)
	\end{array}\right) = 
	\left(\begin{array}{c}
	{\by^1}^\top\\
	\vdots\\
	{\by^n}^\top
	\end{array}\right)\quad\mbox{and} \quad \mathbf{f} = \left(\begin{array}{c}
	f(\bx^1)\\
	\vdots\\
	f(\bx^n)
	\end{array}\right)\,,
\end{equation}
for $\bX = (\bx^1,...,\bx^n) \in D^n$ of distinct sampling nodes and a system of functions $(\eta_k)_{k=1}^{m-1}$. We put $\by^i := (\eta_1(\bx^i),...,\eta_{m-1}(\bx^i))^\top, i=1,...,n$.
The coefficients $c_k$, $k=1,\ldots,m-1$, of the approximant
\begin{equation}
S^m_{\bX}f:=\sum_{k = 1}^{m-1} c_k\, \eta_k\, \label{eq:def_SmX}
\end{equation}
are computed via least squares, see Algorithm 1. 
Note that the mapping $f \mapsto S^m_{\bX}f$ is linear for a fixed set of sampling nodes $\bX = (\bx^1,...,\bx^n) \in D^n.$

We start with a concentration inequality for the spectral norm of a matrix of type \eqref{f0}. 
It turns out that in certain situations the complex matrix $\bL_m:=\bL_m(\bX)\in\C^{n\times(m-1)}$ has full rank with high probability, where $\bX = (\bx^1,...,\bx^n)$ is drawn at random from $D^n$ according to a product measure $\Prob = d\varrho^n$.
We will find that the eigenvalues of 
\begin{equation} \label{defHm}
\mathbf{H}_m:=\mathbf{H}_m(\bX) = \frac{1}{n}\bL_m^\ast \bL_m 
= \frac{1}{n} \sum\limits_{i = 1}^n \by^i \otimes \by^i \in\C^{(m-1)\times(m-1)}\,,
\end{equation}
are bounded away from zero with high probability if $m$ is small enough compared to $n$ and the functions $\eta_k(\cdot)$ denote an orthonormal system with respect to the measure $\varrho$ from which the nodes in $\bX$ are sampled. Let us define the corresponding spectral function 
\begin{equation}
	N(m) := \sup\limits_{\bx \in D}\sum\limits_{k=1}^{m-1}|\eta_k(\bx)|^2\,.
\end{equation}

From \cite[Theorem\ 1.1]{Tr11} we get the following result.
\begin{theorem}[\textit{Matrix Chernoff}] \label{Tropp}
For a finite sequence \((\bA_k)\) of independent, Hermitian, positive semi-definite random matrices with dimension $m$ satisfying \(\lambda_{\max}(\bA_k) \leq R\) almost surely it holds 
\[\Prob\Big(\lambda_{\min}\Big(\sum_{k = 1}^n \bA_k  \Big) \leq (1- t) \mu_{\min}\Big) \leq m \Big(\frac{e^{-t}}{(1-t)^{1-t}} \Big)^{\sfrac{\mu_{\min}}{R}}\]
\[\Prob\Big(\lambda_{\max}\Big(\sum_{k = 1}^n \bA_k  \Big) \geq (1+ t) \mu_{\max}\Big) \leq m \Big( \frac{e^t}{(1+t)^{1+t}}\Big)^{\sfrac{\mu_{\max}}{R}}\]
for \(t \in [0,1]\) where \( \mu_{\min}: = \lambda_{\min}\left(\sum_{k=1}^n \Ept \bA_k \right) \) and \( \mu_{\max}: = \lambda_{\max}\left(\sum_{k=1}^n \Ept \bA_k \right)\).
\end{theorem}

\begin{theorem}
 Let \(n,m \in \N \), \(m\ge2 \) and $\{\eta_1(\cdot),  \eta_2(\cdot), \eta_3(\cdot),\dots,\eta_{m-1}(\cdot)\}$ be an orthonormal system in $L_2(D,\varrho)$. Let $\mathbf{H}_m$ be given as above and \(\bx^1,...,\bx^n\in D\) drawn i.i.d.\ at random according to $\Prob = d\varrho$ we have for $0<t <1$ that
$$
 	\Prob(\lambda_{\min} (\mathbf{H}_m) < 1-t) \leq m \exp\Big(-{\frac{n\, \log c_t}{N(m)}}\Big)\,,
$$
as well as 
$$
 	\Prob(\lambda_{\max} (\mathbf{H}_m) > 1+t) \leq m \exp\Big(-{\frac{n\, \log d_t}{N(m)}}\Big)\,,
$$
where \(c_t :=(1-t)^{1-t}e^t\) and \(d_t :=(1+t)^{1+t}e^{-t}\).
\end{theorem}
\begin{proof}
We apply Theorem \ref{Tropp}. To do this we define \(\bA_i = \frac{1}{n}\, \by^i \otimes \by^i\). One easily sees that all the matrices  $\bA_i$ are always positive semi-definite and \(\lambda_{\min}\Big(\sum_{i=1}^n \Ept \bA_i \Big) =1\). We have that \[\lambda_{\max} (\bA_i) = \|\by^i\|^2/n \leq N(m)/n\,.\]
Plugging this into Theorem \ref{Tropp} yields
\[ \Prob(\lambda_{\min}(\mathbf{H}_m) \leq 1-t)  \leq  m \Big[ \frac{e^{-t}}{(1-t)^{1-t}} \Big]^{n/N(m)} \leq m \exp \Big(-{\frac{n \log c_t}{N(m)}} \Big) . \]
\end{proof}
\begin{theorem} \label{control}
For \(n \geq m\) and \(r > 1\) the matrix \(\mathbf{H}_m\)  has only eigenvalues greater than $ 1/2$ with probability at least \(1 - n^{1-r}\) if 
\begin{equation}\label{f20}
 N(m) \leq \frac{n}{ 7 \,  r \log n}.
\end{equation} 
In particular, we have 
\begin{equation}\label{f100}
	\|(\bL_m^{\ast}\bL_m)^{-1}\bL_m^{\ast}\|_{2\to 2} \leq \sqrt{\frac{2}{n}}\,.
\end{equation}
\end{theorem}
\begin{proof}
Choosing \(t = 1/2\) and solving for $N(m)$ in the above probability bound (using $n^{1-r}$ on the right-hand side) gives the desired result. Indeed
$$
\Prob(\lambda_{\min} (\mathbf{H}_m) < 1-t)  \leq m \exp\Big(-{\frac{n \log c_t}{N(m)}} \Big) \leq n^{1-r}\,.
 $$
This gives the following implications (read from bottom to top)
\begin{equation}
\begin{split}
	\log(m) - \log(c_t) \frac{n}{N(m)} &\leq \log n^{1-r} \\
	\frac{\log m-\log n^{1-r} }{\log c_t} &\leq \frac{n}{N(m)} \\
	N(m) &\leq \frac{n\log c_t}{\log m - \log n^{1-r}} \\
	N(m) &\leq \frac{n}{7 (\log n - (1-r) \log n)}\\
N(m) &\leq \frac{n}{ 7\, r \log n}.
\end{split}
\end{equation}
The bound in \eqref{f100} is a consequence of \cite[Proposition 3.1]{KUV19}.
\end{proof}
From \cite[Proposition 3.1]{KUV19} we also get a lower bound of $\|(\bL_m^{\ast}\bL_m)^{-1}\bL_m^{\ast}\|_{2\to 2}$ with high probability.
\begin{corollary}\label{lower_norm}
Let $\{\eta_1(\cdot),  \eta_2(\cdot), \eta_3(\cdot),...\}$ be an  orthonormal system in $L_2(D,\varrho)$. Let further $r>1$ and $m,n \in \N$, $m\ge 2$ such that \[ N(m) \leq \frac{n}{10 \,  r \log n}\] holds. Then the random matrix $\bL_m$ from \eqref{f0} satisfies
$$
	\|(\bL_m^{\ast}\bL_m)^{-1}\bL_m^{\ast}\|_{2\to 2} \geq \sqrt{\frac{2}{3n}}
$$
with probability at least $1-n^{1-r}$, where the nodes are sampled i.i.d according to $\varrho$\,.
\end{corollary}

\section{Norm concentration for infinite matrices}
\label{sec_prob2}
In this section we want to extend some of the results from Section \ref{sect_prob} to the infinite dimensional framework. We provide a new concentration inequality derived from the non-commutative Khintchine inequality via a bootstrapping argument using a symmetrization result by Ledoux,  Talagrand \cite{LeTa91} for Rademacher sums of random operators \( {\bf B}_i = \by^i \otimes  \by^i\), where \(\by^i\) will denote a complex random infinite \(\ell_2 \)-sequence.
\begin{definition}[\textit{Schatten-p-Norm}]
For a compact operator \({\bf A}: H \to K\) between complex Hilbert spaces $H$ and $K$ and 1 \(\leq p \leq \infty\) we define the Schatten-p-norm: \[\|{\bf A}\|_{S_p} = \|\sigma({\bf A})\|_p \] where \(\sigma(\bf{A})\) is the vector of singular values of \({\bf A}\).
\end{definition}
The quantity $\|\cdot\|_{S_p}$ is a norm, see, e.g. \cite{Di11}.

\begin{corollary} \label{Schatten}
One easily sees that \[ \|{\bf A}\|_{2 \to 2} = \|{\bf A}\|_{S_\infty} \leq \|{\bf A}\|_{S_p}\]
and that for {\bf A} with rank at most \(r\) it holds that \[\|{\bf A}\|_{S_p} \leq r^{\sfrac{1}{p}} \|{\bf A}\|_{2 \to 2} \]  for \(1\leq p \leq \infty\). 
\end{corollary}

\begin{definition}[\textit{Schatten-class}]
Let \(H, K \) be complex Hilbert spaces and \(1 \leq p < \infty\). The \(p\)-th Schatten-class is defined as \[S_p(H,K) : = \big\{ {\bf A}: H \to K, {\bf A} \,\, compact, \|{\bf A} \|_{S_p} < \infty \big\}.\]
\end{definition}

\begin{theorem}[\textit{Non-commutative Khintchine inequality, \cite{Bu01, Bu05}}]
 Let \(n \in \N\) and \({\bf B}_i\), $i=1,...,m$,  denote operators from \( S_{2n}\). Let further \(\varepsilon_i\) denote independent Rademacher variables for \( i= 1 \dots m \). Then it holds
 \[\Ept_{\boldsymbol{\varepsilon}}\Big\|  \sum_{i=1}^m \varepsilon_i \, {\bf B}_i \Big\|_{S_{2n}} ^{2n} \leq \frac{(2n)!}{2^n n!} \max \Big\{ \Big\|\Big(\sum_{i=1}^m  {\bf B}_i {\bf B}_i ^{\ast} \Big)^{\sfrac{1}{2}}\Big\|_{S_{2n}} ^{2n} , \Big\|\Big(\sum_{i=1}^m  {\bf B}_i ^{\ast} {\bf B}_i \Big)^{\sfrac{1}{2}}\Big\|_{S_{2n}} ^{2n}\Big\}. \]
\end{theorem}

\begin{corollary}[\textit{Rudelson's Lemma}] \label{Rudelson}
Let \(\by^i\) be a sequence from the complex \( \ell_2\) and \(\varepsilon_i\) independent Rademacher variables for \( i= 1 \dots m \). Then it holds for \(2 \leq p < \infty\) that \[ \Big( \Ept_{\boldsymbol{\varepsilon}} \Big\|  \sum_{i=1}^m \varepsilon_i \, \by^i \otimes  \by^i \Big\|_{2 \to 2} ^{p} \Big) ^{\sfrac{1}{p}} \leq 2^{\sfrac{3}{4p}} m^{\sfrac{1}{p}}\sqrt{p} e^{-\frac{1}{2}} \sqrt{\Big\|  \sum_{i=1}^m  \, \by^i \otimes  \by^i \Big\|_{2 \to 2} } \, \max_{i = 1 \dots m} \| \by^i \|_{2}. \]
\end{corollary}
\begin{proof}
We utilize the non-commutative Khintchine inequality with \( {\bf B}_i := \by^i \otimes  \by^i\) which belong to every $S_{2n}$ since they have (at most) rank 1. Applying literally the same arguments as in \cite[Lemma 6.18]{Ra10} we obtain the result (see \cite{Mo20} for details).
\end{proof}

We estimate tails of random variables by means of their moments. We will use a well-known relation, see e.g.  \cite[Proposition 6.5]{Ra10}.
\begin{proposition}[Moments and tails] \label{Moments and Tails}
Let \(X\) be a random variable that for all \(p \geq p_0\) satisfies \[ \big(\Ept|X|^p \big)^\frac{1}{p} \leq \alpha \beta^\frac{1}{p} p^\frac{1}{\gamma} \] for some constants \(\alpha, \beta, \gamma, p_0 > 0\). Then \[\Prob\big(|X| \geq e^\frac{1}{\gamma}\,\alpha u \big) \leq \beta e^{-\frac{u^\gamma}{\gamma}}\] for all \(u \geq p_0^{\sfrac{1}{\gamma}}.\)
\end{proposition}

From \cite[Lemma 6.3]{LeTa91} we get the following result.
\begin{proposition}[\textit{Symmetrization}, \cite{LeTa91}] \label{Symmetrization}
Let \(F : \R_+ \to  \R_+ \) be convex. Let \(\bX_i\), $i=1,...n$, be independent random variables in a separable Banach space $(B, \|\cdot \|)$ such that \( \Ept F (\|\bX_i\|) < \infty \). Let further $\boldsymbol{\varepsilon} = (\varepsilon_i)_{i=1}^n$ be independent Rademacher variables which are also independent of \(\bX_i\). 
Then it holds that \[\Ept F\Big( \sup_{f \in D } \Big| f \big( \sum_{i=1}^n \bX_i \big) - \Ept f \big( \sum_{i=1}^n \bX_i \big) \Big| \Big) \leq \Ept F\big( 2 \big\| \sum_{i=1}^n \varepsilon_i \bX_i \big\| \big)\,,\]
where \(D\) is a countable set of linear functionals with \(\|\bx\| = \sup\limits_{f \in D}\big|f(\bx)\big|\) for all \(\bx \in B\).
\end{proposition}

\begin{proposition} \label{main1}
Let \(\by^i , i= 1 \dots n \), be i.i.d random sequences from the complex \( \ell_2\). Let further $n \geq 3$, $r > 1$, $M>0$ such that \(\| \by^i \|_2 \leq M\) for all $i=1 \dots n$ almost surely and \( \Ept  \by^i \otimes  \by^i  =  {\bf \Lambda}\)  for all \(i=1 \dots n\). Then
\[ \Prob \Big( \Big\| \frac{1}{n} \sum_{i=1}^n  \, \by^i \otimes  \by^i - {\bf \Lambda} \Big\|_{2 \to 2} \geq F \Big) \leq 2^\frac{3}{4} \, n^{1-r}\,,\]
where \(F := \max\Big\{ \frac{8r \log n}{n} M^2 \kappa^2 , \| {\bf \Lambda} \|_{2 \to 2} \Big\}\) and \( \kappa = \frac{1+\sqrt{5}}{2}\) .
\end{proposition}
\begin{proof} 
We use a method as in \cite[Theorem 7.3]{Ra10}. 
For \(2 \leq p < \infty\) we put
\[ E_p := \Ept \Big\| \frac{1}{n} \sum_{i=1}^n  \, \by^i \otimes  \by^i - {\bf \Lambda} \Big\|_{2 \to 2}^p\,.\] 
Since \(\sum_{i=1}^n  \, \by^i \otimes  \by^i\) has rank (at most) \(n\) it is compact. The expectation matrix \({\bf \Lambda}\) is a positive semidefinite operator with finite trace since \(\| \by^i \|_2 \leq M\) for all $i=1 \dots n$ almost surely. This means \( {\bf \Lambda}\) is a trace class operator and therefore compact.
Since \( \frac{1}{n} \sum_{i=1}^n  \, \by^i \otimes  \by^i - {\bf \Lambda}\) is compact and the subspace of all compact operators \(K(\ell_2, \ell_2)\) from \(\ell_2\) to \( \ell_2\) is separable we can choose a countable set \(D\) from the dual space of \(K(\ell_2, \ell_2)\) as in Proposition \ref{Symmetrization} such that 
\[ \Big\| \frac{1}{n} \sum_{i=1}^n  \, \by^i \otimes  \by^i - {\bf \Lambda} \Big\|_{2 \to 2} = \sup\limits_{f \in D}\Big|f \Big(\frac{1}{n} \sum_{i=1}^n  \,   \by^i \otimes  \by^i \Big)  -  f \Big( \Ept \frac{1}{n}\sum\limits_{i=1}^n \by^i \otimes  \by^i \Big) \Big|.\]
Since	\(f\) is a continuous linear functional we get 
\begin{equation}\label{compact} E_p = \Ept \Big( \sup\limits_{f \in D}\Big| f \Big(\frac{1}{n} \sum_{i=1}^n  \,  \by^i \otimes  \by^i \Big)  - \Ept f \Big(\frac{1}{n}\sum\limits_{i=1}^n \by^i \otimes  \by^i \Big) \Big|^p \Big) .
\end{equation}
Proposition \ref{Symmetrization} applied to (\ref{compact}) with $F(t) = t^p$ together with Rudelson's lemma (Corollary \ref{Rudelson}) for $2\leq p < \infty$ in the form \[ \Big( \Ept_{\boldsymbol {\varepsilon}} \Big\|  \sum_{i=1}^n \varepsilon_i \, \by^i \otimes  \by^i \Big\|_{2 \to 2} ^{p} \Big) ^{\sfrac{1}{p}} \leq 2^{\sfrac{3}{4p}} n^{\sfrac{1}{p}}\sqrt{p} e^{-\frac{1}{2}} \sqrt{\Big\|  \sum_{i=1}^n  \, \by^i \otimes  \by^i \Big\|_{2 \to 2} } \, \max_{i = 1 \dots n} \| \by^i \|_{2} \] 
yields
\begin{equation*}
\begin{split}
E_p & \leq 2^p \, \Ept_\by  \Ept_{\boldsymbol {\varepsilon}} \Big\| \frac{1}{n} \sum_{i=1}^n \varepsilon_i \, \by^i \otimes  \by^i \Big\|_{2 \to 2} ^{p} \\
 & \leq \Big( \frac{2}{\sqrt{n}} \Big)^p 2^{\frac{3}{4}} n \, p^{\frac{p}{2}} e^{-\frac{p}{2}}\, \Ept_\by \Big( \sqrt{\Big\|  \sum_{i=1}^n  \, \by^i \otimes  \by^i \Big\|_{2 \to 2}} \, \max_{i = 1 \dots n}\| \by^i \|_{2} \Big)^p \\
& \leq \Big( \frac{2}{\sqrt{n}} \Big)^p 2^{\frac{3}{4}} n \, p^{\frac{p}{2}} e^{-\frac{p}{2}} {M}^{p} \, \Ept_\by \Big( \sqrt{\Big\|  \sum_{i=1}^n  \, \by^i \otimes  \by^i \Big\|_{2 \to 2}} \Big)^p \\
& \leq \Big( \frac{2}{\sqrt{n}} \Big)^p 2^{\frac{3}{4}} n \, p^{\frac{p}{2}} e^{-\frac{p}{2}} {M}^{p} \,   \Big( \sqrt{\Ept_\by \Big( \Big\|  \sum_{i=1}^n  \, \by^i \otimes  \by^i - {\bf \Lambda} \Big\|_{2 \to 2}  \Big) + \| {\bf \Lambda} \|_{2\to 2}} \Big)^p \\
\end{split}
\end{equation*}
because of \(\| \by^i \|_2 \leq M \) and H\"older's inequality. Using triangle inequality and the fact that \( \Ept\big(\|\bX\| + \|\bY\| \big)^p \leq \Big(\big(\Ept \|\bX\|^p \big)^{\sfrac{1}{p}} + \big(\Ept \|\bY\| ^p \big)^{\sfrac{1}{p}} \Big)^p \) we have
\[E_p \leq \Big( \frac{2}{\sqrt{n}} \Big)^p 2^{\frac{3}{4}} n \, p^{\frac{p}{2}} e^{-\frac{p}{2}} {M}^{p} \,   \Big( \sqrt{\Ept_\by \Big( \Big\|  \sum_{i=1}^n  \, \by^i \otimes  \by^i - {\bf \Lambda} \Big\|_{2 \to 2} ^{p} \Big)^{\sfrac{1}{p}} + \| {\bf \Lambda}\|_{2\to 2}} \Big). \] 
Setting \(D_{p,n,M} := \frac{2}{\sqrt{n}} 2^{\frac{3}{4p}} M p^{\frac{1}{2}} n^{\frac{1}{p}} e^{-\frac{1}{2}}\) yields
\begin{equation}\label{Ept_bound}
E_p^{\sfrac{1}{p}} \leq  D_{p,n,M} \sqrt{E_p^{\sfrac{1}{p}} + F}\,,
\end{equation} 
where \(F \geq \| {\bf \Lambda} \|_{2\to 2} \) will be chosen later. 
Solving this regarding \(E_p^{\sfrac{1}{p}}\) gives \[E_p^{\sfrac{1}{p}} \leq \frac{D_{p,n,M}^2}{2} + \sqrt{D_{p,n,M}^2 \cdot F + \frac{D_{p,n,M}^4}{4}}.\]
We now consider the random variable \( \min\Big\{F,\frac{1}{n} \Big\|  \sum_{i=1}^n  \, \by^i \otimes  \by^i - {\bf \Lambda} \Big\|_{2 \to 2}\Big\}\). Obviously \[ \Big( \Ept\min\Big\{F,\Big\| \frac{1}{n} \sum_{i=1}^n  \, \by^i \otimes  \by^i - {\bf \Lambda} \Big\|_{2 \to 2}\Big\}^p \Big) ^{\sfrac{1}{p}} \leq \min\{F,E_p^{\sfrac{1}{p}}\}.\]
In the case of \( D_{p,n,M}^2 \leq F\) we get \[\min\{F,E_p^{\sfrac{1}{p}}\} \leq E_p^{\sfrac{1}{p}} \leq D_{p,n,M} \sqrt{F}\Big(\frac{1+\sqrt{5}}{2} \Big) =: D_{p,n,M} \sqrt{F} \, \kappa \]
and otherwise \[\min\{F,E_p^{\sfrac{1}{p}}\} \leq F \leq \ D_{p,n,M} \sqrt{F} \leq D_{p,n,M} \sqrt{F} \, \kappa. \] This yields \[\Big( \Ept\min\Big\{F,\Big\| \frac{1}{n} \sum_{i=1}^n  \, \by^i \otimes  \by^i - {\bf \Lambda} \Big\|_{2 \to 2}\Big\}^p \Big) ^{\sfrac{1}{p}} \leq \kappa  D_{p,n,M}\sqrt{ F}\,. \] 
Using Proposition \ref{Moments and Tails} we get that
\begin{equation}\label{mom->tails}
\Prob \Big( \min\Big\{ F,\Big\| \frac{1}{n} \sum_{i=1}^n  \, \by^i \otimes  \by^i - {\bf \Lambda} \Big\|_{2 \to 2}\Big\} \geq \frac{2}{\sqrt{n}} M \kappa \sqrt{F} u \Big) \leq 2^\frac{3}{4} \, n \exp\Big( \frac{- u^2}{2} \Big)
\end{equation}
 for all \( u \geq \sqrt{2}\). 
Now we choose \(u = \sqrt{2r \log n}\) with \(r > 1\) and \(n \geq 3.\)This gives \[ \Prob \Big( \min\Big\{ F,\Big\| \frac{1}{n} \sum_{i=1}^n  \, \by^i \otimes  \by^i - {\bf \Lambda} \Big\|_{2 \to 2}\Big\} \geq \frac{2}{\sqrt{n}} M \kappa \sqrt{F} \sqrt{2r \log n} \Big) \leq 2^\frac{3}{4} \, n^{1-r}.\] 
In case \( F \geq \frac{2}{\sqrt{n}} M \kappa \sqrt{F} u\) we can avoid the minimum on the left-hand side. It clearly holds
\begin{equation*}
 F  \geq \frac{2}{\sqrt{n}} M \kappa \sqrt{F} u = \frac{2}{\sqrt{n}} M \kappa \sqrt{F} \sqrt{2r \log n},
\end{equation*}
and hence
$$
\sqrt{F} \geq \frac{2}{\sqrt{n}} M \kappa \sqrt{2r \log n}\,.
$$
The latter is satisfied if
\begin{equation}\label{F}
	F := \max\Big\{ \frac{8r \log n}{n} M^2 \kappa^2 , \| {\bf \Lambda}\|_{2 \to 2} \Big\}\,.
\end{equation} 
This yields  \[ \Prob \Big( \Big\| \frac{1}{n} \sum_{i=1}^n  \, \by^i \otimes  \by^i - {\bf \Lambda} \Big\|_{2 \to 2} \geq F \Big) \leq 2^\frac{3}{4} \, n^{1-r}.\]
\end{proof}

Using this result we are now able to proof our main concentration inequality.

\paragraph{Proof of Theorem \ref{int_1}} \begin{proof} Let us return to \eqref{Ept_bound} in the above proof. Since $\|\mathbf{\Lambda}\|_{2\to 2 } \leq 1$ we get as a consequence for $0<t\leq 1$
\begin{equation}\label{Ept_bound2}
E_p^{\sfrac{1}{p}} \leq  \tilde{D}_{p,n,M} \sqrt{E_p^{\sfrac{1}{p}} + t}
\end{equation} 
with 
$$
	\tilde{D}_{p,n,M} := \frac{1}{\sqrt{t}}\frac{2}{\sqrt{n}} 2^{\frac{3}{4p}} M p^{\frac{1}{2}} n^{\frac{1}{p}} e^{-\frac{1}{2}}\,.
$$
We continue in the proof as above using $\tilde{D}_{p,n,M}$ instead of $D_{p,n,M}$ and without replacing $u$ by $\sqrt{2r\log n}$ in \eqref{mom->tails}. With the same argumentation as above we get rid of the minimum and obtain this time 
\begin{equation}\label{bound2}
\Prob \Big( \Big\| \frac{1}{n} \sum_{i=1}^n  \, \by^i \otimes  \by^i - {\bf \Lambda} \Big\|_{2 \to 2} \geq F \Big) \leq 2^\frac{3}{4} n\exp(-u^2/2)
\end{equation}
for $F \geq \max\{4M^2u^2\kappa^2/(nt),t\}$. The maximum is no longer necessary if we choose $u^2:=t^2n/(4M^2\kappa^2)$. Plugging this choice into \eqref{bound2} and noting that $8\kappa^2 \leq 21$ gives the desired bound. \end{proof}

\begin{remark}\label{Oliv_comp} The first result of this type is due to Rudelson \cite{Ru99} for vectors from $\R^n$. Complex versions were proved by Rauhut \cite{Ra10} and Oliveira \cite[Lem.\ 1]{Ol10}. Note that the result stated by Oliveira (Lemma 1) contains a small incorrectness in the probability bound. A corrected version has been stated in \cite[Prop.\ 4.1]{KUV19}. In his paper Oliveira also comments on the infinite dimensional complex situation where $m=\infty$ but does not give a full proof. The proof method is different from ours. Note also that in \cite[Cor.\ 2.6]{PaMe06} Mendelson and Pajor give a concentration result for the infinite dimensional case of real vectors. Let us finally mention that a version of our Theorem \ref{intr1} (in the next section) under more restrictive assumptions has been recently proved by M. Ullrich \cite{Ul20} based on Oliveira's concentration result.
\end{remark}

\section{Reproducing kernel Hilbert spaces}
\label{setting}
We will work in the framework of reproducing kernel Hilbert spaces. The relevant theoretical background can be found in \cite[Chapt.\ 1]{BeTh04} and \cite[Chapt.\ 4]{StChr08}. The papers \cite{HeBo04} and \cite{StSc12} are also of particular relevance for the subject of this paper. 

Let $L_2(D,\varrho_D)$ be the space of complex-valued square-integrable functions with respect to~$\varrho_D$. Here $D \subset \R^d$ is an arbitrary subset and $\varrho_D$ a measure on $D$. 
We further consider a reproducing kernel Hilbert space $H(K)$ with a Hermitian positive definite kernel $K(\cdot,\cdot)$ on $D \times D$. The crucial property of reproducing kernel Hilbert spaces is the fact that Dirac functionals are continuous, or equivalently, the reproducing property holds
$$
		f(\bx) = \langle f, K(\cdot,\bx) \rangle_{H(K)} 
$$
for all $\bx \in D$. It ensures that point evaluations are continuous functionals on $H(K)$. We will use the notation from \cite[Chapt.\ 4]{StChr08}. In the framework of this paper, the finite trace of the kernel 
\begin{equation}\label{integrab}
	\trace{K}:=\|K\|^2_{2} = \int_{D} K(\bx,\bx)d\varrho_D(\bx) < \infty
\end{equation}
or its boundedness  
\begin{equation} \label{CK000}
		\|K\|_{\infty} := \sup\limits_{\bx \in D} \sqrt{K(\bx,\bx)} < \infty
\end{equation}
is assumed. The boundedness of $K$ implies that $H(K)$ is continuously embedded into $\ell_\infty(D)$, i.e., 
\begin{equation}\label{emb0}
		\|f\|_{\ell_{\infty}(D)} \leq  \|K\|_{\infty}\cdot\|f\|_{H(K)}\,.
\end{equation}
With $\ell_{\infty}(D)$ we denote the set of bounded functions on $D$ and with $\|\cdot\|_{\ell_{\infty}(D)}$ the supremum norm. Note, that we do not need the measure $\varrho_D$ for this embedding. In fact, here we mean ``boundedness'' in the strong sense (in contrast to essential boundedness w.r.t. the measure $\varrho_D$). The embedding operator 
\begin{equation}\label{f00b}
\Id:H(K) \to L_2(D,\varrho_D)
\end{equation}
is Hilbert-Schmidt under the finite trace condition \eqref{integrab}, see \cite{HeBo04}, \cite[Lemma 2.3]{StSc12}, which we always assume from now on. We additionally assume that $H(K)$ is at least infinite dimensional. Let us denote the (at most) countable system of strictly positive eigenvalues $(\lambda_j)_{j\in \N}$ of $W_{K,\varrho_D} = \Id_{K,\varrho_D}^{\ast} \circ \Id_{K,\varrho_D}$ arranged in non-increasing order, i.e.,
$$
		\lambda_1 \geq \lambda_2 \geq \lambda_3 \geq \cdots > 0.
$$
We will also need the left and right singular vectors $(e_k)_k \subset H(K)$ and $(\eta_k)_k \subset L_2(D,\varrho_D)$ which both represent orthonormal systems in the respective spaces related by $e_k = \sigma_k \eta_k$ with $\lambda_k = \sigma_k^2$ for $k\in \N$\,. We would like to emphasize that the embedding \eqref{emb0} is not necessarily injective. In other words, for certain kernels there might be a also a nontrivial nullspace of the embedding $\Id$ in \eqref{f00b}. Therefore, the system $(e_k)_k$ from above is not necessarily a basis in $H(K)$. It would be a basis under additional restrictions, e.g., the kernel $K(\cdot,\cdot)$ is continuous and bounded (Mercer kernel).
Based on this observation we will decompose the kernel $K(\cdot,\cdot)$ as follows 
\begin{equation} \label{K(x,x)}
  \begin{split}
		K(\bx,\by) &= K^0(\bx,\by) + K^1(\bx,\by)\\
		&:=\Big(K(\bx,\by)-\sum\limits_{k=1}^{\infty} e_k(\bx)\overline{e_k(\by)}\Big) + \sum\limits_{k=1}^{\infty} e_k(\bx)\overline{e_k(\by)}.
  \end{split}	
\end{equation}
By Bessel's inequality we get that 
\begin{equation}\label{C_K}
	\operatorname{tr}_0(K):=\trace{K^0}=\int_{D} K(\bx,\bx)d\,\varrho_D(\bx) - \sum\limits_{k=1}^{\infty} \lambda_k	\geq 0\,.
\end{equation}
It is shown in \cite{HeBo04}, \cite[Lemma 2.3]{StSc12} that if $\trace{K}<\infty$ and $H(K)$ is separable then $\operatorname{tr}_0(K) = 0$.
As we will see below, it will make a big difference if $\operatorname{tr}_0(K)$ vanishes or not. The second case is only apparent if $H(K)$ is non-separable. In other words, if \(H(K)\) is separable the function \(K^0(\bx,\bx) := K(\bx,\bx)- \sum_{k=1}^\infty |e_k(\bx)|^2\) is zero almost everywhere with respect to the measure \(\varrho_D\). 
Let us finally define the two crucial quantities 
\begin{equation}\label{f1b}
	N(m) := \sup\limits_{\bx \in D}\sum\limits_{k=1}^{m-1}|\eta_k(\bx)|^2\
\end{equation}
and 
\begin{equation}\label{f1bb}
	T(m) := \sup\limits_{\bx \in D}\sum\limits_{k=m}^{\infty}|e_k(\bx)|^2\,.
\end{equation}
The first one is often called ``spectral function'', see \cite{Gr19} and the references therein. 

\section{Sampling recovery guarantees for separable RKHS}
\label{Sect5}
In this section we deal with the case that $H(K)$ is a separable Hilbert space on a subset $D \subset \R^d$ which is compactly embedded in $L_2(D,\varrho_D)$ for a given measure $\varrho_D$. The first Theorem gives a result in a more restrictive situation, namely that $\varrho_D$ is a probability measure and the kernel is bounded. In the second theorem we sample with respect to the probability density function $\varrho_m$ defined below in \eqref{density} and invented by Krieg, Ullrich \cite{KrUl19, KrUl20}. We use the same proof strategy as in \cite{KUV19}. Here we do not apply Rudelson's bound \cite{Ru99} on the expectation. We rather use the concentration inequality proved in Proposition \ref{main1}. This leads to a polynomial decaying failure probability, see also \cite{Ul20}.

\begin{theorem}\label{thm9} Let $H(K)$ be a separable reproducing kernel Hilbert space on a set $D \subset \R^d$ with a positive definite kernel $K(\cdot,\cdot)$ such that $\sup_{\bx \in D} K(\bx,\bx) <\infty$. We denote with $(\sigma_j)_{j\in\N}$ the non-increasing sequence of singular numbers of the embedding $\Id:H(K) \to L_2(D,\varrho_D)$ for a probability measure $\varrho_D$\,.
Let further $r>1$ and $m,n \in \N$, \(n \geq 3\) where $m\ge 2$ is chosen such that 
\begin{equation}\label{f13}
  N(m) \leq \frac{n}{7 r\log n}
\end{equation} 
holds. Drawing $\bX = (\bx^1,...,\bx^n)$ i.i.d.\ at random according to $\varrho_D$, we have 
\[ \Prob \Big( \sup_{\|f\|_{H(K)} \leq 1} \Big\| f - S_{\mathbf{X}}^m f \Big\|_{L_2(D,\varrho_D)}^2 \leq 5 \max\Big\{\sigma_m^2, \frac{8r \log n}{n} T(m) \kappa^2  \Big\}\Big)\geq 1 - \eta n^{1-r}\,, \]
where \(\eta = 2^\frac{3}{4} +1\), \( \kappa = \frac{1+\sqrt{5}}{2}\) and \(N(m), \, T(m)\) are defined in \eqref{f1b}, \eqref{f1bb} .
\end{theorem}

\begin{proof}\sloppy 
We define the events
\begin{align*}
A&:=\Big\{\bX \in D^n~:~\frac{1}{n}\Big\|  \mathbf{\Phi}_m^\ast\mathbf{\Phi}_m\Big\|_{2 \to 2} \leq F + \sigma_m^2 \Big\},\\
B&:=\Big\{\bX\in D^n~:~\frac{1}{2}\leq \lambda_{i}(\mathbf{H}_m), \, i=1 \dots m \Big\}, \\
\end{align*}
where \(F\) appears in \eqref{F} and \(\mathbf{H}_m\) in \eqref{defHm}. The operator \(\mathbf{\Phi}_m\) is given by
\begin{equation*}
		 \mathbf{\Phi}_m:\ell_2 \to \R^n,\quad
		 \mathbf{z} \mapsto \left(\begin{array}{c}
		 \<\mathbf{z},\by^1\>\\
		 \vdots\\
		 \<\mathbf{z},\by^n\>
		 \end{array}\right)	    
\end{equation*}
with \(\by^i = (e_m(\bx^i),e_{m+1}(\bx^i) \dots)^\top\) for all \( i=1 \dots n\). Hence we may choose $F$ as 
$$
	F:=\max\Big\{ \frac{8r \log n}{n} T(m)\kappa^2 , \sigma_m^2 \Big\}
$$
in this specific situation. It clearly holds
$$
 \Prob(A \cap B)= 1 - \Prob(A^\complement \cup B^\complement).
$$
Using the union bound estimate we get
\[ \Prob(A^\complement \cup B^\complement) \leq\Prob(A^\complement) + \Prob(B^\complement).\] Theorem \ref{control} implies \[ \Prob(B^\complement)\leq  n^{1-r}.\]
And after noting \[ \Prob\Big(A^\complement\Big) = \Prob\Big(\frac{1}{n}\Big\| \mathbf{\Phi}_m^\ast\mathbf{\Phi}_m \Big\|_{2 \to 2} > F + \|\mathbf{\Lambda}\|_{2 \to 2}  \Big) \leq \Prob\Big(\Big\| \frac{1}{n}\mathbf{\Phi}_m^\ast\mathbf{\Phi}_m - \mathbf{\Lambda} \Big\|_{2 \to 2} > F  \Big)  \]
 we infer from \(\mathbf{\Phi}_m^\ast\mathbf{\Phi}_m = \sum_{i=1}^n  \, \by^i \otimes  \by^i\) and Proposition \ref{main1} that 
\[\Prob\Big(A^\complement\Big) \leq \Prob\Big( \Big\| \frac{1}{n} \sum_{i=1}^n  \,\by^i \otimes  \by^i - \mathbf{\Lambda} \Big\|_{2 \to 2} \geq F\Big) \leq 2^\frac{3}{4} \, n^{1-r} .\]
In total we have \[\Prob(A \cap B) \geq 1- \eta n^{1-r}\]
with $\eta := 2^{3/4}+1$\,.
According to the proof of \cite[Theorem 5.5]{KUV19} we need the function $T(\cdot,\bx) := K(\cdot,\bx)-\sum_{j=1}^{\infty}e_j(\cdot)\overline{e_j(\bx)}$, which denotes an element in $H(K)$. Its norm is given by
\begin{equation}\label{norm_T}
	\|T(\cdot,\bx)\|_{H(K)}^2 := \langle T(\cdot,\bx), T(\cdot,\bx)\rangle_{H(K)} = K(\bx,\bx)-\sum_{j=1}^{\infty}|e_j(\bx)|^2\,.
\end{equation}
Note that the function in \eqref{norm_T} is zero almost everywhere because of the fact that we have an equality sign in \eqref{C_K} due to our assumptions (separability of $H(K)$ and finite trace). Hence, $\Prob(C) = 1$ with 
$$
	C := \Big\{\bX \in D^n~:~\sum\limits_{k=1}^n \|T(\cdot,\bx^k)\|_{H(K)} = 0\Big\}\,.
$$
Let now $\bX \in A\cap B\cap C$. Then we can use for any $f\in H(K)$ with $\|f\|_{H(K)}\leq 1$ a similar argument as in \cite[Theorem 5.5]{KUV19}
\begin{align*}
	\|f-S^m_\bX f\|^2_{L_2(D,\varrho_D)}
	&\leq \sigma^2_m + \|(\bL_m^{\ast}\bL_m)^{-1}\bL_m^{\ast}\|_{2\to 2}^2\cdot \sum\limits_{k=1}^n \Big|\Big(f-P_{m-1}f\Big)(\bx^k)\Big|^2\\	&= \sigma^2_m + \frac{2}{n}\|\mathbf{\Phi}_m\|_{2\to 2}^2 + \frac{6\|K\|_{\infty}}{n}\sum\limits_{k=1}^n \|T(\cdot,\bx^k)\|_{H(K)}\\
&= \sigma_m^2+\frac{2}{n}\|\mathbf{\Phi}_m\|_{2\to 2}\\
&\leq 2F+3\sigma_m^2\,.	
\end{align*}
This yields
\[  \Prob\Big(\sup\limits_{\|f\|_{{H(K)}}\leq 1} \|f-S^m_\mathbf{X} f\|^2_{L_2(D,\varrho_D)} \leq 2F + 3\sigma_m^2 \Big) \geq 1 - \eta n^{1-r}\]
and therefore
\[ \Prob \Big( \sup_{\|f\|_{H(K)} \geq 1} \Big\| f - S_{\mathbf{X}}^m f \Big\|_{L_2(D,\varrho_D)}^2 \leq 5 \max\Big\{\sigma_m^2, \frac{8r \log n}{n} T(m) \kappa^2  \Big\}\Big) \geq 1-\eta n^{1-r}\,.\] 

\end{proof}

In the sequel we consider a more general situation. The measure where the points are sampled will be adapted to the spectral properties of the embedding. This allows to specify the bound above in terms of the singular numbers of the embedding and benefit from their decay.  
Let us recall the density function from \cite{KrUl19} which we will adapt to our framework as follows
\begin{equation}\label{density}
	\varrho_m(\bx) = \frac{1}{2} \bigg(\frac{1}{m-1} \sum_{j = 1}^{m-1} |\eta_j(\bx)|^2 + \frac{K(\bx,\bx) - \sum_{j = 1}^{m-1}|e_j(\bx)|^2}{\int_D K(\bx,\bx) d\varrho_D(\bx) - \sum_{j = 1}^{m-1}\lambda_j} \bigg).
\end{equation}
\begin{algorithm}[H] 
\caption{Weighted least squares approximation  \cite{CoMi17},\cite{KrUl19},\cite{KUV19}.}\label{algo1:reweighted}
  \begin{tabular}{p{1.2cm}p{4.5cm}p{8.9cm}}
    Input: & $\bX = (\bx^1,...,\bx^n)\in D^n$ \hfill & set of distinct sampling nodes, \\
      & $\mathbf{f} = (f(\bx^1),...,f(\bx^n))^\top$ \hfill & samples of $f$ evaluated at the nodes from $\bX$, \\
      & $m\in\N$ & $m < n$ such that the matrix $\tilde{\bL}_m$ in \eqref{eq:tilde_L} has full (column) rank.
  \end{tabular}
  \begin{algorithmic}
	\STATE
      Compute weighted samples $\boldsymbol{g}:=(g_j)_{j=1}^n$ with $g_j:=\begin{cases}0,  & \varrho_m(\bx^j)=0,\\
      f(\bx^j)/\sqrt{\varrho_m(\bx^j)}, & \varrho_m(\bx^j)\neq 0\,.
       \end{cases}$
  
  	\STATE
  	Solve the over-determined linear system 
  	\begin{equation}
  	\widetilde{\bL}_m \cdot (\tilde{c}_1,...,\tilde{c}_{m-1})^\top = \mathbf{g}\,, \; \widetilde{\bL}_m:=\Big(l_{j,k}\Big)_{j=1,k=1}^{n,m-1},\; l_{j,k}:=\begin{cases}0,  & \varrho_m(\bx^j)=0,\\
  	      \eta_k(\bx^j)/\sqrt{\varrho_m(\bx^j)}, & \varrho_m(\bx^j)\neq 0,
  	       \end{cases}
  	       \label{eq:tilde_L}
  	\end{equation}
  	via least squares (e.g.\ directly or via the LSQR algorithm \cite{PaSa82}), i.e., compute
  	$$
  	(\tilde{c}_1,...,\tilde{c}_{m-1})^\top := (\widetilde{\bL}_m^{\ast}\widetilde{\bL}_m)^{-1} \,\widetilde{\bL}_m^{\ast}\cdot \mathbf{g}.
  	$$
  \end{algorithmic}
   Output:  $\mathbf{\tilde{c}} = (\tilde{c}_1,...,\tilde{c}_{m-1})^\top\in \C^{m-1}$ coefficients of the approximant $\widetilde{S}_{\bX}^m f:=\sum_{k = 1}^{m-1} \tilde{c}_k \eta_k$.
\end{algorithm}
We know from 
\eqref{C_K} that the sequence of singular numbers is square summable.  We use the modified density function $\varrho_m(\cdot):D \to \R$ which has been introduced in \cite{KUV19} as a version of the one from \cite{KrUl19}. 
As above, the family $(e_j(\cdot))_{j\in \N}$ represents the eigenvectors of the non-vanishing eigenvalues of the compact self-adjoint operator $W_{\varrho_D}:=\Id^*\circ\Id : H(K) \to H(K)$, the sequence $(\lambda_j)_{j\in \N}$ represents the ordered eigenvalues and finally $\eta_j:=\lambda_j^{-1/2} e_j$.

Clearly, as a consequence of \eqref{K(x,x)} the function $\varrho_m$ is positive and defined point-wise for any $\bx \in D$. Moreover, it can be computed precisely from the knowledge of $K(\bx,\bx)$ and the first $m-1$ eigenvalues and corresponding eigenvectors. It clearly holds that $\int_D \varrho_m(\bx)d\varrho_D(\bx) = 1$. Here is one of the main results of this paper (note that Theorem \ref{intr1} from the introduction is a simple reformulation of the theorem below).

\begin{theorem}\label{sampling_numbers} Let $H(K)$ be a separable reproducing kernel Hilbert space of complex-valued functions defined on $D$ such that 
$$\int_{D} K(\bx,\bx)d\varrho_D(\bx) < \infty$$
for some measure $\varrho_D$ on $D$, where $(\sigma_k)_k$ denotes the non-increasing sequence of singular numbers of the embedding $\Id:H(K) \to L_2(D,\varrho_D)$. Then we have for $n\in \N$ and 
\begin{equation} \label{choice_m2}
    m := \left\lfloor\frac{n}{14  r\log n }\right\rfloor
\end{equation}
the bound
\[ \Prob \Big( \sup_{\|f\|_{H(K)} \leq 1} \| f - \widetilde{S}_{\bX}^m f\|_{L_2(D,\varrho_D)}^2 \leq 5 \max\Big\{\sigma_m^2, \frac{16r\kappa^2 \log n}{n}\sum\limits_{j=m}^{\infty}\sigma_j^2\Big\}\Big)\geq 1 - \eta n^{1-r}\,, \]
where $\bX$ is sampled i.i.d. according to $\varrho_m(\bx)d\varrho_D(\bx)$ in \eqref{density} above and \(\eta = 2^\frac{3}{4} +1\), \( \kappa = \frac{1+\sqrt{5}}{2}\).
\end{theorem}
\begin{proof} This result is a consequence of Theorem \ref{thm9} above which is applied to the newly constructed probability measure  $\varrho_m(\cdot)$ together with 
\begin{equation}\label{tilde_kernel}
	\widetilde{K}_m(\bx,\by) := \frac{K(\bx,\by)}{\sqrt{\varrho_m(\bx)}\sqrt{\varrho_m(\by)}}\,.
\end{equation}
We observe
$$
\sup_{\|f\|_{H(K)} \leq 1 } \big\| f - \widetilde{S}_{\bX}^m f \big\|_{L_2(D,\varrho_D)} \leq \sup_{\|g\|_{H(\widetilde{K}_m)} \leq 1 } \big\| g - S_{\bX}^m g \big\|_{L_2(D,\mu_m)}\,,
$$
$\widetilde{N}(m) \leq 2(m-1)$, and $\widetilde{T}(m) \leq 2\sum_{j=m}^{\infty}\sigma_j^2$, see \cite[Thms.\ 5.5, 5.8]{KUV19}.  Applying Theorem \ref{thm9} leads to the stated bound. 
\end{proof}

\section{Sampling discretization of the $L_2$-norm}
Motivated from supervised learning theory, see e.g. \cite{CuckerZhou}, one is interested in uniform bounds for the following version of the  ``defect function''
$$
	L_{\bX}(f) := \int_D |f(\bx)|^2\,d\varrho_D(\bx) - \frac{1}{n}\sum\limits_{j=1}^n |f(\bx^j)|^2
$$
with respect to $f$ belonging to some hypothesis space $\mathcal{H}$ which is usually embedded into $\mathcal{C}(D)$, the continuous functions on the domain $D$. From a more classical perspective, authors were interested in discretizing $L_p$-norms using Marcinkiewicz-Zygmund inequalities. This subject has been recently studied systematically by V.N.\ Temlyakov \cite{T2018}, see also K. Gr\"ochenig \cite{Gr19}. The following theorem will be an immediate implication of our concentration result in Theorem \ref{int_1}.

\begin{theorem}\label{f1} Let $\varrho_D$ denote a probability measure on the measurable subset $D \subset \R^d$ and $H(K)$ be a separable reproducing kernel Hilbert space with kernel $K(\cdot,\cdot)$ such that
$$\|K\|_\infty := \sup\limits_{\bx \in D} \sqrt{K(\bx,\bx)}<\infty\,$$
(equivalently, the unit ball in $H(K)$ is uniformly bounded in $\ell_\infty$). Then we have for $0<t\leq 1$
$$
	\Prob\Big( \sup\limits_{\|f\|_{H(K)}\leq 1}\Big|\int_D |f(\bx)|^2\,d\varrho_D(\bx)-\frac{1}{n}\sum\limits_{i=1}^n |f(\bx^i)|^2\Big|>t\|\Id\|^2_{K,2}\Big) \leq 2^{3/4}n\exp\Big(-\frac{nt^2 \|\Id\|^2_{K,2}}{21 \|K\|_{\infty}^2}\Big)\,.
$$
If we fix $r>1$ the above bound can be reformulated as 
$$
	\sup\limits_{\|f\|_{H(K)}\leq 1}\left|\int_D |f(\bx)|^2\,d\varrho_D(\bx)-\frac{1}{n}\sum\limits_{i=1}^n |f(\bx^i)|^2\right| \leq 
	\|\Id\|_{H(K)\to L_2}\|K\|_{\infty}\sqrt{\frac{21r\log n}{n}}
$$
with probability exceeding $1-2n^{1-r}$ and $n$ sufficiently large, namely 
$$
	\frac{n}{\log n} \geq \frac{21 r \|K\|_{\infty}^2}{\|\Id\|_{K,2}^2}\,.
$$
The nodes $\bX = (\bx^1,...,\bx^n)$ are randomly drawn according to the product measure $(d \varrho_D(\bx))^n$.

\end{theorem}

\begin{proof} Fix $f$ with $\|f\|_{H(K)} \leq 1$ and put $M:=\|K\|_{\infty}$. Due to the $L_2$-identity
$$
	f = \sum\limits_{i=1}^\infty \sigma_i \langle f,e_i\rangle \eta_i
$$
we find 
$$
   \int_D |f(\bx)|^2\,d\varrho_D(\bx) = \sum\limits_{i=1}^\infty \sigma_i^2|\langle f,e_i\rangle|^2 = \langle{\hat{\mathbf{f}},\mathbf{\Lambda}\hat{\mathbf{f}}}\rangle_{\ell_2}
$$
with $\hat{\mathbf{f}} = (\langle f, e_k \rangle_{H(K)})_{k\in \N}$ and $\mathbf{\Lambda} = \diag(\sigma^2_1,\sigma^2_2,...)$\,. Note that $\|\mathbf{\Lambda}\|_{2\to 2} = \|\Id\|_{K,2}^2$. Furthermore, putting 
$$
	\by = (e_1(\bx^i),e_2(\bx^i),...,e_k(\bx^i),...)^\top\quad,\quad i=1,...,n\,,
$$
we find 
$$
	\frac{1}{n}\sum\limits_{i=1}^n |f(\bx^i)|^2 = \Big\langle \hat{\mathbf{f}},\Big(\frac{1}{n} \sum\limits_{i=1}^n \by^i \otimes \by^i\Big)\hat{\mathbf{f}}\Big\rangle_{\ell_2}
$$
holds almost surely since $\operatorname{tr}_0(K) = 0$ due to the separability of $H(K)$, see the paragraph after \eqref{C_K}. In fact, the identity fails on a nullset $A \subset D^n$, which is independent of $f$. This follows by using the same arguments as after \eqref{norm_T}\,. Hence,
\begin{equation}\nonumber
  \begin{split}
	\sup\limits_{\|f\|_{H(K)}\leq 1}\left|\int_D |f(\bx)|^2\,d\varrho_D(\bx)-\frac{1}{n}\sum\limits_{i=1}^n |f(\bx^i)|^2\right| 
	&= 	\sup\limits_{\|\hat{\mathbf{f}}\|_{\ell_2}\leq 1}\Big|\langle\hat{\mathbf{f}},\mathbf{\Lambda}\hat{\mathbf{f}} \rangle_{\ell_2} - \Big\langle \hat{\mathbf{f}},\Big(\frac{1}{n} \sum\limits_{i=1}^n \by^i \otimes \by^i\Big)\hat{\mathbf{f}}\Big\rangle_{\ell_2}\Big| \\
	& = \sup\limits_{\|\hat{\mathbf{f}}\|_{\ell_2}\leq 1}\Big|\Big\langle \hat{\mathbf{f}},\Big(\mathbf{\Lambda} - \frac{1}{n} \sum\limits_{i=1}^n \by^i \otimes \by^i\Big)\hat{\mathbf{f}}\Big\rangle_{\ell_2}\Big|\\
	& = \Big\| \mathbf{\Lambda} - \frac{1}{n} \sum\limits_{i=1}^n \by^i \otimes \by^i\Big\|_{2 \to 2}\\
	&\leq t\|\mathbf{\Lambda}\|_{2\to 2}
  \end{split}
\end{equation}
with probability exceeding $1-\exp(-t^2n/(21\tilde{M}^2))$ by Theorem \ref{int_1}. Here $\tilde{M}^2 = M^2/\|\mathbf{\Lambda}\|_{2\to 2}$\,. Hence, we may choose $t = \sqrt{21\tilde{M}^2 r \log(n)/n} \leq 1$ to finally get
\begin{equation}\label{MZ}
		\sup\limits_{\|f\|_{H(K)}\leq 1}\left|\int_D |f(\bx)|^2\,d\varrho_D(\bx)-\frac{1}{n}\sum\limits_{i=1}^n |f(\bx^i)|^2\right| 
		\leq \sqrt{21\|\mathbf{\Lambda}\|_{2\to 2}M^2r\frac{\log n}{n}}\,.
\end{equation}
\end{proof}

\begin{remark} {\em (i)} The uniform boundedness (in $\ell_\infty$) of a function class is sometimes also called $M$-boundedness in the learning theory literature. It represents a common assumption there to analyze the defect function. In fact, uniform  bounds on the defect function are proved by using the concept of covering numbers of the unit ball of $H(K)$ in $\ell_{\infty}(D)$, see \cite{T2018} and \cite{KoTe}. In the above theorem covering number estimates were not used at all. 

{\em (ii)} The quantity $\|K\|_{\infty}$ may be replaced by $\|K\|_{L_{\infty}(D,\varrho_D)}$, i.e., the essential supremum with respect to the probability measure $\varrho_D$. Since $\|K\|_{L_{\infty}(D,\varrho_D)}$ might by smaller than $\|K\|_{\infty}$ we obtain a slight improvement. 
\end{remark}

Now we want to get rid of the uniform boundedness of the class and only assume the finite trace. We have to change the sampling measure and modify the norm discretization operator by incorporating weights. The corresponding theorem reads as follows. 

\begin{theorem}\label{thm_discr2} Let $\varrho_D$ denote an arbitrary measure on the measurable subset $D \subset \R^d$ and $H(K)$ be a separable reproducing kernel Hilbert space with (Hermitian) kernel $K(\cdot,\cdot)$ such that
$$\trace{K} := \int_D K(\bx,\bx)\,d\varrho_D(\bx) < \infty\,.$$
We define the probability density function 
$$
	\nu(\bx) := \frac{K(\bx,\bx)}{\trace{K}}
$$
and sample $\bX = (\bx^1,...,\bx^n)$ from the product measure $(\nu(\bx))d\varrho_D(\bx)^n$.
Then we have 
$$
	\Prob\Big( \sup\limits_{\|f\|_{H(K)}\leq 1}\Big|\int_D |f(\bx)|^2\,d\varrho_D(\bx)-\frac{1}{n}\sum\limits_{i=1}^n \frac{|f(\bx^i)|^2}{\nu(\bx^i)}\Big|>t\|\Id\|_{K,2}^2\Big) \leq 2^{3/4}n\exp\Big(-\frac{nt^2\|\Id\|^2_{K,2}}{21\trace{K} }\Big)\,.
$$
If we fix $r>1$ then this result can be reformulated as 
$$
	\sup\limits_{\|f\|_{H(K)}\leq 1}\left|\int_D |f(\bx)|^2\,d\varrho_D(\bx)-\frac{1}{n}\sum\limits_{i=1}^n \frac{|f(\bx^i)|^2}{\nu(\bx^i)}\right| \leq 
	\sqrt{21\trace{K}\|\Id\|^2r\frac{\log n}{n}}
$$
with probability exceeding $1-2n^{1-r}$ and $n$ sufficiently large, namely 
$$
	\frac{n}{\log n} \geq \frac{21 r \trace{K}}{\|\Id\|_{K,2}^2}\,.
$$
\end{theorem}

\begin{proof} We want to apply Theorem \ref{f1}. Let us define the normalized kernel
$$
	\tilde{K}(\bx,\by) := \frac{K(\bx,\by)}{\sqrt{\nu(\bx)}\sqrt{\nu(\by)}}.
$$
Then $\|\tilde{K}\|_{\infty} = \mbox{tr(K)}$ and 
\begin{equation}
  \begin{split}
	&\sup\limits_{\|f\|_{H(K)}\leq 1}\Big|\int_D |f(\bx)|^2\,d\varrho_D(\bx)-\frac{1}{n}\sum\limits_{i=1}^n \frac{|f(\bx^i)|^2}{\nu(\bx^i)}\Big|\\
	&~~~~~~~= \sup\limits_{\|f\|_{H(\tilde{K})}\leq 1}\Big|\int_D |f(\bx)|^2\,\nu(\bx)d\varrho_D(\bx)-\frac{1}{n}\sum\limits_{i=1}^n |f(\bx^i)|^2\Big|\,.
\end{split}
\end{equation}
It remains to note that 
$$\|\Id\|_{K,2,d\varrho_D} = \|\Id\|_{\tilde{K},2,\nu(\cdot)d\varrho_D(\cdot)}$$
and we may apply Theorem \ref{f1}.
\end{proof}
\section{Non-separable RKHS}
\label{nonsepRKHS}
Now we deal with a more general situation and drop the separability assumption for $H(K)$. We only assume the finite trace property \eqref{f00}. For this purpose we define the new density function
\begin{equation}\label{f101b}
\varrho_m(\bx) = \frac{1}{3} \bigg( \frac{ \sum_{j = 1}^{m-1} |\eta_j(\bx)|^2}{m-1}  + \frac{ \sum_{j=m}^\infty |e_j(\bx)|^2}{\sum_{j=m}^{\infty} \lambda_j}  + \frac{K_0(\bx,\bx)}{ \trace{K^0}}  \bigg) 
\end{equation}
and the operator $\widetilde{S}^m_{\bX}$ from Algorithm 1. Clearly, it holds \(\int\varrho_m(\bx) d\varrho_D(\bx) = 1\). Theorem \ref{non-sep} provides the bound 
\begin{equation}\label{f200} 
\sup_{\|f\|_{H(K)} \leq 1 } \big\| f - \widetilde{S}_{\bX}^m f \big\|^2_{L_2(D,\varrho_D)}  \leq C  \max\Big\{\sigma_m^2, \frac{r \log n }{n} \sum_{j=m}^\infty \sigma_j^2, \frac{\trace{K^0}}{n}  \Big\}
\end{equation}
with an absolute constant $C>0$ and $m := \lfloor n/(14 r\log n)\rfloor$. Note that this result improves on a result in Wasilkowski, Wo{\'z}niakowski \cite{WaWo01}, see also Novak, Wo{\'z}niakowski \cite[Thm.\ 26.10]{NoWoIII}. The authors in \cite[Thm.\ 26.10]{WaWo01} constructed a recovery operator using $n$ samples having squared worst-case error not greater than
$$
	\min\Big\{\sigma_{\ell}^2 + \frac{\trace{K}\ell}{n}~:~\ell = 1,2,3,...\Big\}.
$$
If we for instance assume that $\sum_{k=1}^{\infty} \sigma_k^p < \infty$ for some $0<p\leq 2$ then we may balance $\ell \asymp n^{p/(p+2)}$ to obtain a rate of $\mathcal{O}(n^{-1/(p+2)})$. In Theorem \ref{non-sep}, see \eqref{f200}, we obtain a rate of $o(\sqrt{(\log n)/n})$ already for $p=2$. In case $p<2$ we obtain $\mathcal{O}(n^{-1/2})$. It seems that, in general,  the decay properties of the singular values have a rather weak influence on the recovery bound compared to the separable case, where it is much better than $\mathcal{O}(n^{-1/2})$. A lower bound showing that we can not essentially improve on $\mathcal{O}(n^{-1/2})$ with the above algorithm will be provided in \cite{Mo20}.   

\paragraph{Proof of Theorem \ref{non-sep}}
\begin{proof}{\em Step 1.} Let us assume that $M:=\|K\|_{\infty} = \sup_{\bx \in X}\sqrt{K(\bx,\bx)} < \infty$.
By the spectral theorem we can decompose \[H(K) = N(\Id) \oplus \overline{ \spn \{e_1(\cdot), e_2(\cdot), \dots\}}\]
where $N$ is the nullspace of the embedding. Let us now define 
\[K^1(\bx,\by) = \sum_{j=1}^\infty e_j(\bx) \overline{e_j(\by)} \] and \[K^0(\bx,\by) = K(\bx,\by) - K^1(\bx,\by) .\] Therefore, \(K^0(\bx,\by) \) is the reproducing kernel of the nullspace $N(\Id)$ and \(\sup_{\bx \in X} \sqrt{K^0(\bx, \bx)} =: M_0 \leq M < \infty\). We estimate
\begin{equation} \label{anonsep2}
\begin{split}
& \sup_{\|f\|_{H(K)} \leq 1 } \big\| f - S_{\bX}^m f \big\|_{L_2(D,\varrho_D)} \\
& \leq \sup_{\|g\|_{H(K^0)} \leq 1 } \big\| g - S_{\bX}^m g \big\|_{L_2(D,\varrho_D)} + \sup_{\|g\|_{H(K^1)} \leq 1 } \big\| g - S_{\bX}^m g \big\|_{L_2(D,\varrho_D)}.
\end{split}
\end{equation}
The second summand can be treated with Theorem \ref{thm9}. The space \(H(K^1)\) is separable and 
\(K^1(\bx,\bx)\) is bounded. 
Theorem \ref{thm9} gives
\begin{equation} \label{anonsep4}
\sup_{\|g\|_{H(K^1)} \leq 1} \Big\| g - S_{\bX}^m g \Big\|_{L_2(D,\varrho_D)}^2 \leq 5 \max\Big\{\sigma_m^2, \frac{8r \log n}{n} T(m) \kappa^2 \Big\}
\end{equation}
with probability at least \( 1 - \eta n^{1-r} \) whenever  \eqref{f13} holds. The number $\kappa$ is the same as in Proposition \ref{main1}. 

Note that all the functions in \(H({K}^0)\) are zero in \(L_2(D,\varrho_D)\) since this space is the nullspace of the embedding. Hence
\begin{equation*}
\begin{split}
\sup_{\|g\|_{H({K}^0)} \leq 1 } \big\| g - {S}_{\bX}^m g \big\|_{L_2(D,\varrho_D)}^2 = & \sup_{\|g\|_{H({K}^0)} \leq 1 } \big\| {S}_{\bX}^m g \big\|_{L_2(D,\varrho_D)}^2 \\
\leq &\frac{2}{n} \sup_{\|g\|_{H({K}^0)} \leq 1 } \sum_{i=1}^n |g(\bx^i)|^2
\end{split}
\end{equation*}
holds for the same \(\bX = (\bx^1, \dots,  \bx^n) \) for which \eqref{anonsep4} holds. We only need the operator norm of \(\mathbf{H}_m\) (see \eqref{defHm}) to be larger than \(\frac{1}{2}\) which comes from Theorem \ref{Tropp}. At this point we employ the ``Representer Theorem'' from learning theory, see for instance \cite[Theorem 5.5]{StChr08}. We claim that
\begin{equation} \label{anonsep5}
\sup_{\|g\|_{H({K}^0)} \leq 1 } \sum_{i=1}^n |g(\bx^i)|^2 = \sup_{\boldsymbol{\alpha}^\intercal {K}^0[\bX] \overline{\boldsymbol{\alpha}} \leq 1} \sum_{i=1}^n|{g}(\bx^i)|^2\,,
\end{equation}
where \(\big({K}^0(\bx^i, \bx^j)\big)_{i,j = 1}^n\) is the kernel taken at the points from $\bX$ and \( {g}(\bx) = \sum_{i =1}^n \alpha_i K^0(\bx,\bx^i)\) \\
In other words, we can reduce the problem to the finite dimensional space\\\( \spn\big\{ K^0(\cdot, \bx^1),\dots, K^0(\cdot, \bx^n) \big\}.\) Note that 
\begin{equation*} 
\begin{split}
\|{g}\|_{H({K}^0)}^2 & = \sum_{i = 1}^n \sum_{j = 1}^n {K}^0(\bx^i, \bx^j) \alpha_i \overline{\alpha_j} \\
& = \boldsymbol{\alpha}^\intercal {K}^0[\bX]  \overline{\boldsymbol{\alpha}} \\
& \leq 1.
\end{split}
\end{equation*}
The reason is that \(g \in H(K^0)\) can be decomposed into \(g = g_1 + g_2\) with \(g_1 \perp g_2\) and \(g_1 = Pg \), the orthogonal projection onto  \(\spn\big\{ K^0(\cdot, \bx^1), \dots ,K^0(\cdot, \bx^n) \big\}\). Due to \(g_1 \perp g_2\), we have that \(\langle g_2, K^0(\cdot, \bx^i)\rangle  = 0 = g_2(\bx^i)\) for all $i$.
Hence \(\sum_{i=1}^n |g(\bx^i)|^2 = \sum_{i=1}^n |g_1(\bx^i)|^2\) and \( \|g_1\|_{H({K}^0)}^2 \leq 1 \).
Therefore 
\begin{align} 
\sup_{\|g\|_{H({K}^0)} \leq 1 } \big\| g - {S}_{\bX}^m g \big\|_{L_2(D,\varrho_D)}^2 & \leq \frac{2}{n} \sup_{\|g\|_{H({K}^0)} \leq 1 } \sum_{i=1}^n |g(\bx^i)|^2 \notag \\
& \leq \frac{2}{n} \sup_{\boldsymbol{\alpha}^\intercal {K}^0[\bX]  \overline{\boldsymbol{\alpha}}   \leq 1 } \sum_{i=1}^n \big| \sum_{j=1}^n \alpha_j {K}^0(\bx^i , \bx^j) \big|^2\,.  \label{anonsep6} 
\end{align} 
Since \(H({K}^0)\) is the nullspace of the embedding we know that \({K}^0(\bx^i , \bx^j)\) is zero almost surely for \(i \neq j\). We can therefore continue to estimate \eqref{anonsep6} by
\begin{equation}  \label{anonsep4a}
\begin{split}
\frac{2}{n} \sup_{\alpha^\intercal K[\bX]  \overline{\alpha}   \leq 1 } \sum_{i=1}^n \big| \sum_{j=1}^n \alpha_j  {K}^0(\bx^i , \bx^j) \big|^2 & = \frac{2}{n} \sup_{\boldsymbol{\alpha}^\intercal K[\bX] \overline{\boldsymbol{\alpha}}  \leq 1 } \sum_{i=1}^n  |\alpha_i|^2  \big| {K}^0(\bx^i , \bx^i) \big|^2 \\
& \leq \frac{2 M_0^2}{n} \sup_{\boldsymbol{\alpha}^\intercal K[\bX]  \overline{\boldsymbol{\alpha}}   \leq 1 } \sum_{i=1}^n  |\alpha_i|^2  \big| {K}^0(\bx^i , \bx^i) \big| \\ 
& \leq \frac{2 M_0^2}{n}
\end{split}
\end{equation}
since we have almost surely 
\begin{align*} 
\sup_{\boldsymbol{\alpha}^\intercal {K}^0[\bX]  \overline{\boldsymbol{\alpha}}   \leq 1 } \sum_{i=1}^n  |\alpha_i|^2  \big| {K}^0(\bx^i , \bx^i) \big|  & =  \sum_{i = 1}^n \sum_{j = 1}^n  \alpha_i \overline{\alpha_j}  {K}^0(\bx^i, \bx^j) \\ 
& = \boldsymbol{\alpha}^\intercal {K}^0[\bX]  \overline{\boldsymbol{\alpha}} \\
& \leq 1.
\end{align*} 
This leads to
\begin{equation}\label{f15}
\begin{split}
 \sup_{\|f\|_{H(K)} \leq 1 } \big\| f -S_{\bX}^m f \big\|^2_{L_2(D,\varrho_D)}& \leq (\sup_{\|g\|_{H(K^0)} \leq 1 } \big\| g - S_{\bX}^m g \big\|_{L_2(D,\varrho_D)} + \sup_{\|g\|_{H(K_m^1)} \leq 1 } \big\| g - S_{\bX}^m g \big\|_{L_2(D,\varrho_D)})^2 \\
& \leq \Big(\sqrt{\frac{1}{4 \kappa^2}} + \sqrt{5} \Big)^2 \max\Big\{\sigma_m^2, \frac{8r \log n}{n} T(m) \kappa^2,\frac{8 M_0^2 \kappa^2}{n} \Big\}\, \\
& \leq 7 \max\Big\{\sigma_m^2, \frac{8r \log n}{n} T(m) \kappa^2, \frac{8 M_0^2 \kappa^2}{n}  \Big\}\,
\end{split}
\end{equation}
with probability exceeding $1-\eta n^{1-r}$\,.

{\em Step 2.} We define the measure \(d\mu_m(\bx) = \varrho_m(\bx)d\varrho_D(\bx) \) as well as the kernel \(\widetilde{K}_m(\bx,\by)\) as in \eqref{tilde_kernel} and \(\widetilde{K}_m^0, \widetilde{K}_m^1\) accordingly. This gives 
\begin{equation} \label{nonsep1}
\sup_{\|f\|_{H(K)} \leq 1 } \big\| f - \widetilde{S}_{\bX}^m f \big\|_{L_2(D,\varrho_D)} \leq \sup_{\|g\|_{H(\widetilde{K}_m)} \leq 1 } \big\| g - \widetilde{S}_{\bX}^m g \big\|_{L_2(D,\mu_m)}.
\end{equation}
We apply the results from Step 1 to the right-hand side. Hence, we have to know the bound for $\widetilde{K}^0_m$, $\widetilde{N}(m)$ and $\widetilde{T}(m)$, where the latter quantities are associated to $\widetilde{K}^1_m$.
We will now show that \(\widetilde{K}_m^0\) can be bounded by \(3\trace{K^0}\). In fact,
\begin{align} 
\widetilde{K}_m^0(\bx,\bx) & = \frac{K_m^0(\bx,\bx)}{\varrho_m(\bx)} \notag \\
& = \frac{K(\bx,\bx) - \sum_{k = 1}^{\infty} |e_k(\bx)|^2}{\frac{1}{3} \bigg( \frac{ \sum_{j = 1}^{m-1} |\eta_j(\bx)|^2}{m-1}  + \frac{ \sum_{j=m}^\infty |e_j(\bx)|^2}{\sum_{j=m}^{\infty} \lambda_j}  + \frac{K^0(\bx,\bx)}{ \trace{K^0}}  \bigg)} \notag \\
& \leq 3 \, \trace{K^0}. \label{nonsep3b}
\end{align}
Hence, we have $\widetilde{M}_0^2 = 3\trace{K^0}$. By the same arguments as in the proof of Theorem \ref{sampling_numbers}, see also \cite[Thm.\ 5.7]{KUV19}, we get 
$\widetilde{N}(m) \leq 3(m-1)$ and $\widetilde{T}(m) \leq 3\sum_{j=m}^\infty
\sigma_j^2$. Plugging this into \eqref{f15} gives
\begin{align*} 
	\sup_{\|g\|_{H(\widetilde{K}_m)} \leq 1 } \big\| g - \widetilde{S}_{\bX}^m g \big\|^2_{L_2(D,\mu_m)} & \leq 7 \max\Big\{\sigma_m^2, \frac{24 \kappa^2 r \log n }{n} \sum_{j=m}^\infty \sigma_j^2, \frac{24  \kappa^2\trace{K^0}}{n}  \Big\} \\
  \sup_{\|f\|_{H(K)} \leq 1 } \big\| f - \widetilde{S}_{\bX}^m f \big\|^2_{L_2(D,\varrho_D)} &  \leq 441  \max\Big\{\sigma_m^2, \frac{r \log n }{n} \sum_{j=m}^\infty \sigma_j^2, \frac{\trace{K^0}}{n}  \Big\}\,.
\end{align*}
\end{proof}

What concerns a counterpart of the \(L_2\)-norm discretization result for general RKHS having finite trace we can prove the following.

\begin{theorem}
Let \(H(K)\) be a RKHS and \(\varrho_D\) be a measure on $D \subset \R^d$. If
\begin{itemize}
\item[(i)] If \(\|K\|_{\infty}:= \sup_{\bx \in D} \sqrt{K(\bx, \bx)} < \infty\) and  \(\varrho_D\) denotes a probability measure on $D$, where $\bx^i, i=1,...,n$, are drawn i.i.d. according to $\varrho_D$ then
 \[\sup\limits_{\|f\|_{H(K)}\leq 1}\Big| \int_D |f(\bx)|^2 d \varrho_D(\bx) - \frac{1}{n} \sum_{i = 1}^n |f(\bx^i)|^2 \Big| \leq  8\sqrt{r\frac{\log(n)}{n}} \|K\|_{\infty}^2 \]
holds with probability at least \(1 - 2 n^{1-r}\) if $n$ is large enough, i.e.	\(\frac{n}{\log n} \geq \frac{21 r \|K\|_{\infty}^2}{\|\Id\|_{K,2}^2}\).
\item[(ii)] If \(\trace{K} = \int_D K(\bx,\bx) d\varrho_D(\bx) < \infty\) then
 \[\sup\limits_{\|f\|_{H(K)}\leq 1}\Big| \int_D |f(\bx)|^2 d \varrho_D(\bx) - \frac{1}{n} \sum_{i = 1}^n \frac{|f(\bx^i)|^2}{\nu(\bx^i)} \Big| \leq 8 \trace{K}\sqrt{r\frac{\log(n)}{n}}  \] 
holds with probability at least \(1 - 2 n^{1-r}\) and \(\frac{n}{\log n} \geq  \frac{21 r \trace{K}}{\|\Id\|_{K;2}^2}\)
where \(\nu(\bx) = \frac{K(\bx, \bx)}{\trace{K}}\) and $\bx^i, i=1,...,n$, are drawn i.i.d. according to \(\nu(\bx)d\varrho_D(\bx)\).
\end{itemize}
\end{theorem}

\begin{proof}
Since we may have that \(\mbox{tr}_0(K) > 0\) the decomposition \(K(\bx, \by) = K^0(\bx, \by) + K^1(\bx, \by)\) leads to a ``non-trivial'' Kernel \(K^0(\bx, \by)\).
We estimate in case (i):
\begin{align}
\sup_{\|f\|_{H(K)} \leq 1} & \Big| \int_D |f(\bx)|^2 d \varrho_D(\bx)  - \frac{1}{n} \sum_{i = 1}^n |f(\bx^i)|^2 \Big| \notag \\
& = \sup_{\|f\|_{H(K)} \leq 1} \Big| \int_D |f_0(\bx) + f_1(\bx)|^2 d \varrho_D(\bx) - \frac{1}{n} \sum_{i = 1}^n |f_0(\bx^i) + f_1(\bx^i)|^2 \Big| \notag \\
& = \sup_{\|f\|_{H(K)} \leq 1} \Big| \int_D |f_1(\bx)|^2d \varrho_D(\bx) - \frac{1}{n} \sum_{i = 1}^n |f_0(\bx^i) |^2 - \frac{2}{n} \sum_{i = 1}^n |f_0(\bx^i)  f_1(\bx^i)| - \frac{1}{n} \sum_{i = 1}^n |f_1(\bx^i)|^2\Big| \notag \\
& \leq \sup_{\|f\|_{H(K^1)} \leq 1} \Big| \int_D  |f_1(\bx)|^2 d \varrho_D(\bx) - \frac{1}{n} \sum_{i = 1}^n |f_1(\bx^i) |^2\Big| \label{eins} \\
& + \sup_{\|f\|_{H(K^0)} \leq 1} \frac{1}{n} \sum_{i = 1}^n |f_0(\bx^i) |^2 \label{zwei} \\
& + \sup_{\|f\|_{H(K)} \leq 1} \frac{2}{n} \Big(\sum_{i = 1}^n |f_0(\bx^i) |^2\Big)^{\frac{1}{2}} \Big(\sum_{i = 1}^n |f_1(\bx^i) |^2\Big)^{\frac{1}{2}} \label{drei}
\end{align}
and note that \eqref{eins} \(\leq \sqrt{21 \|\Id\|^2 M^2 r \frac{\log{n}}{n}}\) by Theorem \ref{f1} with probability at least \(1-2n^{1-r}\), where $M:=\|K\|_{\infty}$. To estimate \eqref{zwei} we use the same reasoning leading to \eqref{anonsep5} and get 
\begin{equation} \label{ast}
\sup_{\|f\|_{H(K^0)} \leq 1} \frac{1}{n} \sum_{i = 1}^n |f_0(\bx^i) |^2 \leq \frac{1}{n} M^2\,,
\end{equation}
where we used \(K(\bx,\bx) \leq M^2\). We also use \eqref{ast} in order to estimate \eqref{drei}. It holds
\begin{align*}
\sup_{\|f\|_{H(K)} \leq 1} \frac{2}{n} \Big(\sum_{i = 1}^n |f_0(\bx^i) |^2\Big)^{\frac{1}{2}} \Big(\sum_{i = 1}^n |f_1(\bx^i) |^2\Big)^{\frac{1}{2}} & \leq \frac{2}{\sqrt{n}} M \Big( \frac{1}{n} \sum_{i = 1}^n |f_1(\bx^i) |^2\Big)^{\frac{1}{2}} \\
& \leq \frac{2 M^2}{\sqrt{n}}\,. 
\end{align*}
In total we estimate
\begin{align*} \eqref{eins} + \eqref{zwei} + \eqref{drei} & \leq \sqrt{21 r} M^2 \sqrt{\frac{\log(n)}{n}} + \frac{M^2}{n} + \frac{2 M^2}{\sqrt{n}} \\
& \leq 8 M^2  \sqrt{\frac{r\log(n)}{n}}\,.
\end{align*}
To prove (ii) we use the same technique as in Theorem \ref{thm_discr2} replacing \(\frac{1}{n} \sum_{i = 1}^n |f(\bx^i)|^2\) with \(\frac{1}{n} \sum_{i = 1}^n \frac{|f(\bx^i)|^2}{\nu(\bx^i)}\)  where \(\nu(\bx) = \frac{K(\bx, \bx)}{\trace{K}}\)  and also \(M^2 \) by \(\trace{K}\) we can reduce everything to case (i).
\end{proof}

\paragraph{Acknowledgment.} T.U.\ would like to acknowledge support by the DFG Ul-403/2-1. The authors would like to thank T. K\"uhn, V.N. Temlyakov, M. Ullrich, K. Pozharska and two anonymous referees for comments and further references.

\end{document}

%% file: macros.tex
\usepackage[outer=2.5cm,top=2.5cm,bottom=2.5cm,inner=2.0cm]{geometry}

\usepackage{xfrac}
\usepackage{hyperref} 
\usepackage{amsmath}
\usepackage{amsthm}
\usepackage{amssymb}
\usepackage{amsfonts}
\usepackage{dsfont}
\usepackage{stmaryrd}
\usepackage{commath}
\usepackage{mathrsfs}
\usepackage{enumitem}
\usepackage{graphicx}
\usepackage{tikz}
\usepackage{algorithm}
\usepackage{algorithmic}
\usepackage{pgfplots}
\usepackage{subfig}
\usepackage[numbers]{natbib}

\usepackage{todonotes}
\usepackage{placeins}

\font\msbm=msbm10

\numberwithin{equation}{section}

\theoremstyle{plain}
\newtheorem{theorem}{Theorem}[section]

\newtheorem{corollary}[theorem]{Corollary}
\newtheorem{proposition}[theorem]{Proposition}

\theoremstyle{definition}
\newtheorem{definition}[theorem]{Definition}

\newtheorem{remark}[theorem]{Remark}
\def\mathbb#1{\hbox{\msbm{#1}}}

\newcommand{\field}[1]{\ensuremath{\mathds{#1}}}

\newcommand{\N}{\field N}
\newcommand{\Ept}{\field E}\newcommand{\Prob}{\field P}

\newcommand{\C}{\field{C}}

\newcommand{\Z}{\field Z}

\newcommand{\R}{\field R}

\newcommand{\Id}{\ensuremath{\mathrm{Id}}}

\newcommand{\bx}{\mathbf{x}}
\newcommand{\by}{\mathbf{y}}

\newcommand{\bL}{\mathbf{L}}

\newcommand{\bX}{\mathbf{X}}
\newcommand{\bY}{\mathbf{Y}}

\newcommand{\bA}{\mathbf{A}}

\newcommand{\<}{\langle}
\renewcommand{\>}{\rangle}

\newcommand{\be}{\begin{equation}}
\newcommand{\ee}{\end{equation}}
\newcommand{\beq}{\begin{eqnarray}}
\newcommand{\beqq}{\begin{eqnarray*}}
\newcommand{\eeq}{\end{eqnarray}}
\newcommand{\eeqq}{\end{eqnarray*}}

\newcommand{\trace}[1]{\operatorname{tr}(#1)}

\DeclareMathOperator{\spn}{span}

\newcommand{\diag}{\mathrm{diag}}

%% file: Moeller_Ullrich_Rev.bbl
\begin{thebibliography}{10}

\bibitem{BeTh04}
A.~Berlinet and C.~Thomas-Agnan.
\newblock {\em Reproducing kernel {H}ilbert spaces in probability and
  statistics}.
\newblock Kluwer Academic Publishers, Boston, MA, 2004.
\newblock With a preface by Persi Diaconis.

\bibitem{Bu01}
A.~Buchholz.
\newblock Operator {K}hintchine inequality in non-commutative probability.
\newblock {\em Ann. of Math.}, 319:1--16, 2001.

\bibitem{Bu05}
A.~Buchholz.
\newblock {Optimal constants in Khintchine type inequalities for Fermions,
  Rademachers and $q$-Gaussian operators}.
\newblock {\em Bulletin Polish Acad. Sci. Math.}, 53(3):315--321, 2005.

\bibitem{StChr08}
A.~Christmann and I.~Steinwart.
\newblock {\em Support Vector Machines}.
\newblock Springer, 2008.

\bibitem{CoMi17}
A.~Cohen and G.~Migliorati.
\newblock Optimal weighted least-squares methods.
\newblock {\em SMAI J. Comput. Math.}, 3:181--203, 2017.

\bibitem{CuckerZhou}
F.~Cucker and D.~X. Zhou.
\newblock {\em Learning Theory. An Approximation Theory Viewpoint}.
\newblock Cambridge University Press, 2007.

\bibitem{Di11}
S.~Dirksen.
\newblock {\em Noncommutative and vector-valued Rosenthal inequalities}.
\newblock Dissertation, Delft Institute of Applied Mathematics, 2011.

\bibitem{DuTeUl19}
D.~{D}\~ung, V.~N. {T}emlyakov, and T.~{U}llrich.
\newblock {\em {H}yperbolic {C}ross {A}pproximation}.
\newblock Advanced Courses in Mathematics. CRM Barcelona.
  Birkh\"auser/Springer, 2019.


\bibitem{Gr19}
K.~Gr\"{o}chenig.
\newblock Sampling, {M}arcinkiewicz-{Z}ygmund inequalities, approximation, and
  quadrature rules.
\newblock {\em J. Approx. Theory}, 257, 2020.

\bibitem{HeBo04}
M.~Hein and O.~Bousquet.
\newblock Kernels, associated structures and generalizations.
\newblock Technical Report 127, Max Planck Institute for Biological
  Cybernetics, T{\"u}bingen, Germany, 2004.

\bibitem{KUV19}
L.~K{\"a}mmerer, T.~Ullrich, and T.~Volkmer.
\newblock Worst-case recovery guarantees for least squares approximation using
  random samples.
\newblock {\em arXiv:1911.10111v3}, 2019.

\bibitem{KoTe}
S.~V. Konyagin and V.~N. Temlyakov.
\newblock The entropy in learning theory. {E}rror estimates.
\newblock {\em Constr. Approx.}, 25(1):1--27, 2007.

\bibitem{KrUl19}
D.~Krieg and M.~Ullrich.
\newblock Function values are enough for ${L}_2$-approximation.
\newblock {\em Found. Comp. Math.}, to appear, {\em arXiv:math/1905.02516v3}.

\bibitem{KrUl20}
D.~Krieg and M.~Ullrich.
\newblock Function values are enough for ${L}_2$-approximation. Part II. 
\newblock {\em J. Complexity}, to appear, {\em arXiv:2011.01779}.


\bibitem{LeTa91}
M.~Ledoux and M.~Talagrand.
\newblock {\em Probability in Banach Spaces}.
\newblock Springer, 1991.

\bibitem{PaMe06}
S.~Mendelson and A.~Pajor.
\newblock On singular values of matrices with independent rows.
\newblock {\em Bernoulli}, 12:761--773, 2006.

\bibitem{Mo20}
M.~Moeller.
\newblock Norm-concentration results for infinite random matrices with
  independent rows.
\newblock Bachelor's thesis, Faculty of Mathematics, TU Chemnitz, 2020.

\bibitem{NSU20}
N.~Nagel, M. Sch\"afer, T.~Ullrich.
\newblock A new upper bound for sampling numbers.
\newblock {\em Found. Comp. Math.}, to appear, {\em arXiv:2010.00327}.

\bibitem{NoWoIII}
E.~Novak and H.~Wo\'{z}niakowski.
\newblock {\em Tractability of multivariate problems. {V}olume {III}:
  {S}tandard information for operators}, volume~18 of {\em EMS Tracts in
  Mathematics}.
\newblock European Mathematical Society (EMS), Z\"{u}rich, 2012.

\bibitem{Ol10}
R.~I. Oliveira.
\newblock Sums of random {H}ermitian matrices and an inequality by {R}udelson.
\newblock {\em Electr. Comm. Probab.}, 15:203--212, 2010.

\bibitem{PaSa82}
C.~C. Paige and M.~A. Saunders.
\newblock {LSQR}: {A}n algorithm for sparse linear equations and sparse least
  squares.
\newblock {\em ACM Trans. Math. Software}, 8:43--71, 1982.

\bibitem{Ra10}
H.~Rauhut.
\newblock Compressive sensing and structured random matrices.
\newblock In M.~Fornasier, editor, {\em Theoretical Foundations and Numerical
  Methods for Sparse Recovery}, volume~9 of {\em Radon Series on Computational
  and Applied Mathematics}. de Gruyter, Berlin, 2010.

\bibitem{Ru99}
M.~Rudelson.
\newblock Random vectors in the isotropic position.
\newblock {\em J. Funct. Anal.}, 64:60--72, 1999.


\bibitem{StSc12}
I.~Steinwart and C.~Scovel.
\newblock Mercers theorem on general domains: On the interaction between
  measures, kernels, and {RKHS}s.
\newblock {\em Constr. Approx.}, 35, 2012.

\bibitem{T2018}
V.~Temlyakov.
\newblock Sampling discretization error of integral norms for function classes.
\newblock {\em Journal of Complexity}, 54, 2019.

\bibitem{Te18}
V.~N. Temlyakov.
\newblock The {M}arcinkiewicz-type discretization theorems.
\newblock {\em Constr. Approx.}, 48(2):337--369, 2018.


\bibitem{Te20_2}
V.~N. Temlyakov.
\newblock On optimal recovery in ${L}_2$.
\newblock {\em J. Complexity}, to appear.


\bibitem{Tr11}
J.~Tropp.
\newblock User-friendly tail bounds for sums of random matrices.
\newblock {\em Found. Comp. Math.}, 12(4):389--434, 2011.

\bibitem{Ul20}
M.~Ullrich.
\newblock On the worst-case error of least squares algorithms for
  {L}2-approximation with high probability.
\newblock {\em Journal of Complexity}, 60, 2020.

\bibitem{Wa84}
G.~W. Wasilkowski.
\newblock Some nonlinear problems are as easy as the approximation problem.
\newblock {\em Comput. Math. Appl.}, 10(4-5):351--363 (1985), 1984.


\bibitem{WaWo01}
G.~W. Wasilkowski and H.~Wo\'{z}niakowski.
\newblock On the power of standard information for weighted approximation.
\newblock {\em Found. Comput. Math.}, 1(4):417--434, 2001.

\end{thebibliography}
